\documentclass[11pt, twoside,leqno]{article}

\usepackage{amssymb}
\usepackage{amsmath}
\usepackage{amsthm}
\usepackage{yhmath}
\usepackage{mathrsfs}
\usepackage{color}
\usepackage{txfonts}
\usepackage{indentfirst}
\usepackage{booktabs}
\usepackage{graphicx}
\usepackage{subfigure}
\usepackage{tikz}
\usepackage{enumerate}
\usepackage{anysize}
\usepackage{setspace}
\usepackage{float}
\usepackage{color}
\usepackage{chngcntr}

\allowdisplaybreaks

\def\zz{\mathbb{Z}}

\def\nn{\mathbb{N}}

\def\sgn{{\mathop\mathrm{\,sgn\,}}}

\def\ls{\lesssim}

\def\1bf{\mathbf{1}}
\def\0bf{\mathbf{0}}

\def\red{\color{red}}

\newtheorem{theorem}{Theorem}[section]
\newtheorem{lemma}[theorem]{Lemma}
\newtheorem{corollary}[theorem]{Corollary}

\theoremstyle{definition}
\newtheorem{remark}[theorem]{Remark}

\newtheorem{definition}[theorem]{Definition}
\renewcommand{\appendix}{\par
   \setcounter{section}{0}%
   \setcounter{subsection}{0}%
   \setcounter{subsubsection}{0}%
   \gdef\thesection{\@Alph\c@section}%
   \gdef\thesubsection{\@Alph\c@section.\@arabic\c@subsection}%
   \gdef\theHsection{\@Alph\c@section.}%
   \gdef\theHsubsection{\@Alph\c@section.\@arabic\c@subsection}%
   \csname appendixmore\endcsname
 }

\numberwithin{equation}{section}

\begin{document}

\arraycolsep=1pt

\title{\bf\Large
Fractional Fourier Transforms
Meet Riesz Potentials and Image Processing
\footnotetext{\hspace{-0.35cm} 2020 {\it Mathematics Subject
Classification}.
Primary 26A33; Secondary   42B35, 42A38, 94A08.
\endgraf {\it Key words and phrases.} fractional Fourier
transform, fractional Riesz potential,
numerical picture simulation, image encryption, image decryption.
\endgraf This project is partially supported by
the National Natural Science Foundation of China
(Grant Nos.\ 12071197, 12071052 and 11971058) and the
National Key Research and Development Program of China
(Grant No. 2020YF A0712900).}}
\author{Zunwei Fu, Yan Lin, Dachun Yang\footnote{Corresponding
author, E-mail: \texttt{dcyang@bnu.edu.cn}/{\red February 24,
2023}/Final version.}  \ \  and Shuhui Yang}
\date{}
\maketitle

\vspace{-0.8cm}

\begin{center}
\begin{minipage}{13cm}
{\small{\textbf{Abstract}}\quad}
Via chirp functions from fractional Fourier transforms, the authors
introduce fractional Riesz potentials related to chirp functions,
establish their relations
with fractional Fourier transforms, fractional Laplace operators
related to chirp functions, and fractional Riesz transforms related to chirp functions, and
obtain their boundedness  on rotation invariant spaces related to chirp functions. Finally,
the authors give the numerical image simulation of fractional
Riesz potentials related to chirp functions and their applications in image processing.
The main novelty of this article is to propose a new image
encryption method for the double phase coding based on the
fractional Riesz potential related to chirp functions. The symbol of fractional Riesz
potentials related to chirp functions essentially provides greater degrees of freedom and greatly
makes the information more secure.
\end{minipage}
\end{center}
\vspace{0.2cm}


\section{Introduction}\label{sec1}

The image processing   has always been an important topic of
 information sciences, which plays an important role in applied   sciences and hence attracts a lot of  attention (see, for instance, \cite{bplx21,elx13,gcn21,gcn21-2,hcjn22,lzox15,rwdl2019,s14,wn21,wscy2022}).

On the other hand, it is well known that the Fourier transform is one of the most basic important tools
in both pure and applied mathematics. It is also a standard and
powerful tool for analyzing and processing stationary signals,
but it is limited in processing and analyzing non-stationary
signals. In what follows, we use $\mathscr{S}(\mathbb{R}^n)$  to
denote the set of all Schwartz functions equipped with well-known
topology determined by a countable family of norms, and also
$\mathscr{S}'(\mathbb{R}^n)$ the set of all continous linear
functionals on $\mathscr{S}(\mathbb{R}^n)$ equipped with the
weak$-\ast$ topology.

\begin{definition}\label{def-ft}
For any $f\in\mathscr{S}(\mathbb{R}^n)$, its \emph{Fourier
transform} $\widehat{f}$ or $\mathcal{F}(f)$ is defined by
setting, for any $\boldsymbol{\xi} \in \mathbb{R}^n$,
$$\widehat{f}({\boldsymbol{\xi}}):=\mathcal{F}(f):=
\int_{\mathbb{R}^n}f({\boldsymbol{x}})e^{-2\pi
i{\boldsymbol{x}}\cdot{\boldsymbol{\xi}}}\,d{\boldsymbol{x}}$$
and the \emph{inverse Fourier transform} $f^{\vee}$ of $f$
is defined by setting, for any $\boldsymbol{x} \in
\mathbb{R}^n$, $f^{\vee}(\boldsymbol{x}):=\widehat{f}
(-\boldsymbol{x})$.
\end{definition}

The fractional Fourier transform (for short, FRFT) is proposed
and developed to analyze and process non-stationary signals.
At present, the  FRFT has been applied in many fields, such as
partial differential equation (see, for instance,
\cite{l1993}), wavelet analysis (see, for instance,
\cite{yn2003}),   complex transmission (see, for instance,
\cite{rbm1994}),
frequency filter (see, for instance,
\cite{tlw2010}), time-frequency analysis (see, for instance,
\cite{asf17,hs22,hs222,mun2001}),
  optical signal processing (see, for instance,
\cite{bs1994,np2003,oa1995,ozk2001,sds2011}), and optical image
processing (see, for instance,
\cite{dsp2001,lllwl2013}).

The FRFT originated in Wiener's work in \cite{w1929}. In 1980,
the FRFT was given by Namias in \cite{n1980}, which is mainly based
on the eigenfunction expansion method. The integral expressions
of the FRFT on $\mathscr{S}(\mathbb{R})$ and $L^2(\mathbb{R})$
were given by McBride-Kerr in \cite{mk1987} and Kerr in \cite{k1988},
respectively. In 2021, the behavior of the FRFT on $L^p(\mathbb{R})$
for any $p\in[1,2)$ was established by Chen et al. in \cite{cfgw2021}.

In recent years, the research of the multidimensional FRFT
has attracted more and more attention. In \cite{kr2020},
Kamalakkannan and Roopkumar gave the  following definition of the
multidimensional FRFT.

\begin{definition}\label{def-fft}
Let $\boldsymbol{\alpha}
:=(\alpha_1,\ldots,\alpha_n)\in \mathbb{R}^n$  and  $f\in L^1(\mathbb{R}^n)$.
The \emph{multidimensional fractional Fourier transform}
(for short, multidimensional FRFT) $\mathcal{F}
_{\boldsymbol{\alpha}}(f)$, with order $\boldsymbol{\alpha}$, of $f$
 is defined by setting, for any $\boldsymbol{u}\in \mathbb{R}^n$,
$$\mathcal{F}_{\boldsymbol{\alpha}}(f)(\boldsymbol{u}):=
\int_{\mathbb{R}^n}f(\boldsymbol{x})K_{\boldsymbol{\alpha}}
(\boldsymbol{x},\boldsymbol{u})\,d\boldsymbol{x},$$
where, for any $\boldsymbol{x}:=(x_1,\ldots,x_n),
\boldsymbol{u}:=(u_1,\ldots,u_n)\in \mathbb{R}^n$,
$$K_{\boldsymbol{\alpha}}(\boldsymbol{x},\boldsymbol{u})
:=\prod^n_{k=1}K_{\alpha_k}(x_k,u_k)$$ and, for any
$k\in \{1,\ldots,n\}$,
\begin{equation*}
K_{\alpha_k}(x_k,u_k):=\left\{
\begin{array}
[c]{ll}
{c(\alpha_k)} e^{2\pi i\{a(\alpha_k)[x_k^2+u_k^2-2b
(\alpha_k)x_ku_k]\}}
&\quad \mathrm{if}\ \alpha_k\notin\pi \mathbb{Z} ,\\
\delta(x_k-u_k) &\quad \mathrm{if}\ \alpha_k\in2\pi\mathbb{Z},\\
\delta(x_k+u_k) &\quad \mathrm{if}\ \alpha_k\in2\pi\mathbb{Z}+\pi ,
\end{array}
\right.
\end{equation*}
with $a(\alpha_k):=
\frac{\cot(\alpha_k)}{2}:=\frac{\cos(\alpha_k)}{2\sin(\alpha_k)}$, $b(\alpha_k):=\sec(\alpha_k):=\frac{1}{\cos(\alpha_k)}$,
$c(\alpha_k):=\sqrt{1-i\cot(\alpha_k)}$, and $\delta$ being
the Dirac measure at $0$.
\end{definition}

We refer the reader to \cite{kr2020,krz21,krz22,z2018,z2019,zl22} for some studies on the
multidimensional FRFT.

\begin{remark}\label{rem-fft-ft}
Let $\boldsymbol{\alpha}:=(\alpha_1, \ldots,\alpha_n)\in
\mathbb{R}^n$ with $\alpha_k\notin \pi \mathbb{Z}$ for any
$k\in\{1,\ldots,n\}$. The \emph{chirp function}
$e_{\boldsymbol{\alpha}}$ is defined by setting, for any
$\boldsymbol{x}:=(x_1, \ldots,x_n)\in \mathbb{R}^n$,
$$e_{\boldsymbol{\alpha}}(\boldsymbol{x}):=e^{2\pi
i\sum^n_{k=1}a(\alpha_k)x^2_k},$$
where, for any $k\in \{1,\ldots,n\}$, $a(\alpha_k):=
\frac{\cot(\alpha_k)}{2}$.
The chirp function is the most common nonstationary signal
in which the frequency increases (upchirp) or decreases
(downchirp) with time. It is easy to show that the
multidimensional {\rm FRFT} $\mathcal{F}_{\boldsymbol{\alpha}}(f)$
of $f$ can be expressed into that, for any
$\boldsymbol{u}:=(u_1,\ldots,u_n)\in \mathbb{R}^n$,
\begin{equation}\label{eq-1.1}
\mathcal{F}_{\boldsymbol{\alpha}}(f)
(\boldsymbol{u})=c(\boldsymbol{\alpha})e_{\boldsymbol{\alpha}}
(\boldsymbol{u})\mathcal{F}(e_{\boldsymbol{\alpha}}f)
(\widetilde{\boldsymbol{u}}),	
\end{equation}
where  $c(\boldsymbol{\alpha}):=c(\alpha_1 )\cdots c(\alpha_n)$
with $\{c(\alpha_k)\}_{k=1}^n$ the same as in Definition
\ref{def-fft} and where
$\widetilde{\boldsymbol{u}}:=(u_1\csc(\alpha_1),\ldots,u_n\csc(\alpha_n))$ with $\csc(\alpha_k):=\frac{1}{\sin(\alpha_k)}$ for any $k\in \{1,\ldots,n\}$.
Using (\ref{eq-1.1}), we can easily prove that
$\mathcal{F}_{\boldsymbol{\alpha}}(f)$ maps
$\mathscr{S}(\mathbb{R}^n)$ to $ \mathscr{S}(\mathbb{R}^n)$.
Obviously, we can rewrite that, for any $\boldsymbol{x}:=(x_1,\ldots,x_n),
\boldsymbol{u}:=(u_1,\ldots,u_n)\in \mathbb{R}^n$,
$$K_{\boldsymbol{\alpha}}(\boldsymbol{x},\boldsymbol{u})
={c(\boldsymbol{\alpha})}e_{\boldsymbol{\alpha}}
(\boldsymbol{x})e_{\boldsymbol{\alpha}}(\boldsymbol{u})
e^{-2\pi i\sum^{n}_{k=1}x_ku_k\csc(\alpha_k)}.$$
Thus, the multidimensional FRFT is closely related to chirp functions.
\end{remark}

\begin{remark}
When $\boldsymbol{\alpha}=(\frac{\pi}{2}+2k_1\pi,
\ldots,\frac{\pi}{2}+2k_n\pi)$ with
$\{k_j\}_{j=1}^n\subset \mathbb{Z}$, the multidimensional
FRFT goes back to the classical Fourier transform.
\end{remark}

In this article, via chirp functions from fractional Fourier transforms, the authors
introduce fractional Riesz potentials related to chirp functions,
establish their relations
with fractional Fourier transforms, fractional Laplace operators
related to chirp functions, and fractional Riesz transforms related to chirp functions, and
obtain their boundedness  on rotation invariant spaces related to chirp functions. Finally,
the authors give the numerical image simulation of fractional
Riesz potentials related to chirp functions and their applications in image processing.
The main novelty of this article is to propose a new image
encryption method for the double phase coding based on the
fractional Riesz potential related to chirp functions. The symbol of fractional Riesz
potentials related to chirp functions essentially provides greater degrees of freedom and greatly
makes the information more secure.

The remainder of this article is organized as follows.

In Section \ref{sec2},
based on    the multidimensional FRFT, we introduce
the fractional Riesz potential  related to chirp functions and establish the
relations among the multidimensional  FRFT,
the fractional Riesz potential related to chirp functions, the fractional Riesz
transform related to chirp functions, and the Laplace operator related to chirp functions
in  $\mathscr{S}'(\mathbb{R}^n)$. We show that, in the
rotation invariant space,
the boundedness of the fractional Riesz potential is
equivalent to the boundedness of the classical Riesz
potential. In addition, through the relation between
the fractional Fourier transform and the Laplace operator
related to chirp functions, we introduce the fractional
Laplace operator related to chirp functions, which, together with
the fractional Riesz potential, provides  a theoretical basis of
the application of image encryption.

An electronic image simulation of
the fractional Riesz potential related to chirp functions in the two-dimensional
case is given in Section \ref{sec3}.

In Section
\ref{sec4}, we present an image encryption method of double
phase coding based on fractional Riesz potentials related to chirp functions, which
mainly changes the amplitude in the fractional Fourier
domain. Compared with the image encryption method of double
phase coding based on the FRFT, the symbol of fractional
Riesz potentials related to chirp functions in the image encryption method of double phase
coding based on fractional Riesz potentials related to chirp functions essentially provides more
degrees of freedom and greatly improves the security of
information.

A conclusion is given in Section \ref{sec5}.

 The  fractional pseudo-differential operator and the  fractional generalized Sobolev space and other function spaces related to chirp functions as well as their applications will be presented in a forthcoming article.

Finally, we make some conventions on notation.
Let  $\nn:=\{1,2,\ldots\}$    and $\zz_+:=\{0\}\cup\nn$.
For any given multi-index
$\boldsymbol{\zeta}:=(\zeta_1,\ldots,\zeta_n)\in\zz_+^n:=(\zz_+)^n$,
let
 $$|\boldsymbol{\zeta}|:=\zeta_1+\cdots+\zeta_n\ \mathrm{and}\ \partial^{\boldsymbol{\zeta}}
:=\left(\frac{\partial}{\partial x_1}\right)^{\zeta_1}\cdots\left(\frac{\partial}{\partial x_n}\right)^{\zeta_n}.$$
We use $C$ to denote a positive constant
which is independent of the main parameters involved, whose value may vary from line to line.
The \emph{symbol} $g\ls h$ means $g\le Ch$.
  For a complex number $z$ with ${\rm Re}\ z > 0$, let the \emph{gamma function}
$$\Gamma(z):=\int_0^\infty t^{z-1}e^{-t}\,dt.$$
  For any $p\in(0,\infty)$, the \emph{Lebesgue space} $L^p(\mathbb{R}^n)$ is defined to be the set of all the measurable functions $f$ on $\mathbb{R}^n$ such that $$\|f\|_{L^p(\mathbb{R}^n)}:=\left[\int_{\mathbb{R}^n}|f(x)|^p\,dx\right]^{1/p}<\infty.$$
Moreover, when we prove a theorem or the like, we always use the same
symbols in the wanted proved theorem or the like.

\section{Fractional Riesz potentials related to chirp functions}\label{sec2}

Recall that,  for any  $f\in \mathscr{S}(\mathbb{R})$, the \emph{Hilbert transform} $H(f)$ of $f$,  is defined by setting, for any $x\in \mathbb{R}$,
$$H(f)(x):=\frac{1}{\pi}\ {\rm p}.{\rm v}.\int_\mathbb{R}
\frac{f(y)}{x-y}\,dy,$$
which is the prototype of  Calder\'on--Zygmund    operators and plays an irreplaceable role in harmonic analysis. The Hilbert
transform is also a multiplier operator, that is, for any  $f\in \mathscr{S}(\mathbb{R})$ and $x\in \mathbb{R}$,
\begin{equation}\label{eq-H}
\mathcal{F}(Hf)(x)=-i \sgn(x)\mathcal{F}(f)(x).
\end{equation}
It can be seen from (\ref{eq-H}) that the Hilbert transform
 is a phase-shift converter that multiplies the positive
 frequency portion of the original signal by $-i$; in other
 words, it maintains  the same amplitude and shifts the phase
 by $-\pi/2$, while the negative frequency portion is shifted
 by $\pi/2$. Note that the Riesz transform is a natural
generalization of the Hilbert transform in the $n$-dimensional
case and is also a Calder\'on--Zygmund    operator, with properties
analogous to those of the Hilbert transform on
$\mathbb{R}$. For any   $j\in \{1,\ldots, n\}$ and $f\in \mathscr{S}(\mathbb{R}^n)$ , the \emph{Riesz transform} $R_j(f)$   of $f$ is defined by setting,  for any $\boldsymbol{x}\in \mathbb{R}^n$,
\begin{equation*}
R_j(f)(\boldsymbol{x}):=c_n\ {\rm p}.{\rm v}.\int_{\mathbb{R}^n}
\frac{x_j-y_j}{|\boldsymbol{x}-\boldsymbol{y}|^{n+1}}
f(\boldsymbol{y})\,d\boldsymbol{y},
\end{equation*}
where $c_n:={\Gamma(\frac{n+1}{2})}/{\pi^{\frac{n+1}{2}}}$.
The Riesz transform is also a multiplier operator, that is, for any   $j\in \{1,\ldots, n\}$, $f\in \mathscr{S}(\mathbb{R}^n)$, and $\boldsymbol{x}:=(x_1,\ldots,x_n)\in \mathbb{R}^n$,
$$\mathcal{F}(R_jf)(\boldsymbol{x})=-\frac{ix_j}
{|\boldsymbol{x}|}\mathcal{F}(f)(\boldsymbol{x});$$
therefore, the multiplier of the Riesz transform is $-ix_j/|\boldsymbol{x}|$,
and hence the Riesz transform is not only a phase-shift converter,
but also an amplitude attenuator.

In \cite{z1998}, Zayed originally introduced the fractional Hilbert
transform related to chirp functions which has been widely used in signal processing
(see, for instance, \cite{tlw2010,tlw2008,vs2014}). In \cite{cfgw2021},
Chen et al. regarded the
fractional Hilbert transform as the fractional Fourier multiplier
operator. Due to the development and the wide application of the
multidimensional FRFT,  motivated
by the relations among the Fourier transform, the Riesz
transform, the multidimensional  FRFT, and the fractional
Hilbert transform, Fu et al. \cite{fglwy2021} introduced the following fractional Riesz
transform related to chirp functions.

\begin{definition}\label{def-frt} For any $j\in\{1,\ldots, n\}$  and
$\boldsymbol{\alpha}:=(\alpha_1, \ldots,
\alpha_n)\in  \mathbb{R}^n$ with
$\alpha_k\notin \pi \mathbb{Z}$ for any $k\in\{1,\ldots,n\}$, the
\emph{$j$th fractional Riesz transform related to chirp functions}, $R_j^{\boldsymbol{\alpha}}(f)$, of
$f\in \mathscr{S}(\mathbb{R}^n)$ is defined  by setting, for any
$\boldsymbol{x}\in \mathbb{R}^n$,
\begin{align*}
R_j^{\boldsymbol{\alpha}}(f)(\boldsymbol{x}):=c_n\
{\rm p}.{\rm v}.\ e_{-\boldsymbol{\alpha}}(\boldsymbol{x})
\int_{\mathbb{R}^n}\frac{x_j-y_j}{|\boldsymbol{x}
-\boldsymbol{y}|^{n+1}}f(\boldsymbol{y})e_{ \boldsymbol{\alpha}}
(\boldsymbol{y})\,d\boldsymbol{y},
\end{align*}
where $c_n:=\Gamma(\frac{n+1}{2})/\pi^{\frac{n+1}{2}}$ and
     $e_{\boldsymbol{\alpha}} $ is the same as in Remark  \ref{rem-fft-ft}.
\end{definition}

The fractional Riesz transform related to chirp functions, $R_j^{\boldsymbol{\alpha}}$, is
also a fractional multiplier operator. That is, for any  $j\in\{1,\ldots, n\}$, $\boldsymbol{\alpha}:=(\alpha_1, \ldots,
\alpha_n)\in  \mathbb{R}^n$ with
$\alpha_k\notin \pi \mathbb{Z}$ for any $k\in\{1,\ldots,n\}$,
 $f\in \mathscr{S}(\mathbb{R}^n)$, and  $\boldsymbol{u}:=(u_1 ,\ldots, u_n)\in
\mathbb{R}^n$,  one has
\begin{align}\label{eq-R}
\mathcal{F}_{\boldsymbol{\alpha}}\left(R_j^{\boldsymbol{\alpha}}
f\right)(\boldsymbol{u})
=-i\frac{\widetilde{u}_j}{|\widetilde{\boldsymbol{u}}|}
\mathcal{F}_{\boldsymbol{\alpha}}\left(f\right)(\boldsymbol{u}),
\end{align}
where
   $\widetilde{\boldsymbol{u}}:=(\widetilde{u}_1,\ldots,\widetilde{u}
_n)=(u_1\csc(\alpha_1),\ldots, u_n\csc(\alpha_n))$.

In \cite{fglwy2021}, Fu et al. also gave the application of the
fractional Riesz transform related to chirp functions in the edge detection. Compared
with the classical Riesz transform, the fractional Riesz
transform related to chirp functions can detect the information in any direction by
adjusting its order.

Recall that, for any $\beta\in(0,n)$
and $f \in \mathscr{S}(\mathbb{R}^n)$, the \emph{Riesz potential}
is  defined by setting, for any
$\boldsymbol{x}\in \mathbb{R}^n$,
\begin{equation}\label{eq-beta}
	I_\beta(f)(\boldsymbol{x}):=\frac{1}{\gamma(\beta)}
\int_{\mathbb{R}^n}\frac{f(\boldsymbol{y})}{|\boldsymbol{x}
-\boldsymbol{y}|^{n-\beta}}\,d\boldsymbol{y},
\end{equation}
with $\gamma(\beta):=\pi ^{\frac{n}{2}}2^\beta
\Gamma(\frac{\beta}{2})/\Gamma(\frac{n-\beta}{2})$.
The Riesz potential, also known as the fractional integral
operator, is not only a crucial  integral operator in
Fourier analysis, but also plays a very significant role
in fractional differential equations. In this article, in order to avoid  the
confusion with the fractional   operator,
we always  call it the Riesz potential. It is well known that,
for any   $\beta\in(0,n)$,
  $f \in \mathscr{S}(\mathbb{R}^n)$,   and $\boldsymbol{x}\in \mathbb{R}^n$, by the property
of the Fourier transform, one has
\begin{align}\label{eq-RP}
\mathcal{F}(I_\beta f)(\boldsymbol{x})=(2\pi)^{-\beta}
|\boldsymbol{x}|^{-\beta}\mathcal{F}(f)(\boldsymbol{x})
\end{align}
in  $\mathscr{S}'(\mathbb{R}^n)$.

Inspired by the relation between the Riesz potential and the
Fourier transform, as well as the definitions of both the multidimensional  FRFT and
the fractional Riesz transform related to chirp functions, we introduce  a   new fractional Riesz potential related to chirp functions as follows.

\begin{definition}{\label{def-frp}}
If $\beta\in(0,n)$ and
$\boldsymbol{\alpha}:=(\alpha_1, \ldots,
\alpha_n)\in  \mathbb{R}^n$ with
$\alpha_k\notin \pi \mathbb{Z}$ for any $k\in\{1,\ldots,n\}$,
the \emph{fractional Riesz potential related to chirp functions}, $I^{\boldsymbol{\alpha}}_\beta $,
is defined by setting, for any $f\in \mathscr{S}(\mathbb{R}^n)$  and
$\boldsymbol{x}\in \mathbb{R}^n$,
$$I^{\boldsymbol{\alpha}}_\beta(f)(\boldsymbol{x}):=
\frac{1}{\gamma(\beta)}e_{-\boldsymbol{\alpha}}(\boldsymbol{x})
\int_{\mathbb{R}^n}\frac{f(\boldsymbol{y})}{|\boldsymbol{x}
-\boldsymbol{y}|^{n-\beta}}e_{\boldsymbol{\alpha}}
(\boldsymbol{y})\,d\boldsymbol{y},$$
where $e_{\boldsymbol{\alpha}}$ is the same as in Remark  \ref{rem-fft-ft} and $\gamma(\beta)$ is the same as in    \eqref{eq-beta}.
\end{definition}

\begin{remark}
When $\boldsymbol{\alpha}=
(\frac{\pi}{2}+k_1\pi,\ldots,
\frac{\pi}{2}+k_n\pi)$ with $\{k_j\}_{j=1}^n\subset \mathbb{Z} $,
the fractional Riesz potential related to chirp functions goes back to  the
classical Riesz potential.
\end{remark}

The Riesz potential is closely related to the Laplace operator, the Fourier
transform, and the Riesz transform. Based on this, in
  Subsection \ref{sec2.1}, we establish the relations among the fractional Riesz potential,
the multidimensional  FRFT, the fractional Riesz transform related to chirp functions,   and the Laplace operator related to
chirp functions. In Subsection \ref{sec2.2}, we obtain the
boundedness in the rotation invariant space of the fractional
Riesz potential related to chirp functions.

\subsection{Relations among fractional Fourier transforms,
fractional Riesz potentials related to chirp functions, fractional Laplace operators related to
chirp functions,  and fractional Riesz transforms related to chirp functions}\label{sec2.1}

First, we establish the relation between the multidimensional  FRFT and the fractional Riesz
potential related to chirp functions.

\begin{definition}{\rm (see \cite{g12014})}\label{def-fts}  	
   The \emph{Fourier transform}
$\widehat{u}$ of any tempered distribution $u$ is defined
by  setting, for any $f \in \mathscr{S}(\mathbb{R}^n)$,
$\langle\widehat{u},f\rangle:=\langle u,\widehat{f}\rangle.$
\end{definition}

\begin{definition} \label{def-ffts}  	
Let   $\boldsymbol{\alpha}:=(\alpha_1 ,\ldots, \alpha_n )\in
\mathbb{R}^n$ with
$\alpha_k \notin \pi \mathbb{Z}$ for any $k\in \{1,\ldots,n \}$.  The \emph{multidimensional}  FRFT
$\widehat{u}$ of any tempered distribution $u$ is defined
by  setting, for any $f \in \mathscr{S}(\mathbb{R}^n)$,
$\langle\mathcal{F}_{\boldsymbol{\alpha}}u,f\rangle:=\langle u, \mathcal{F}_{\boldsymbol{\alpha}}f\rangle.$
\end{definition}

 Let   $\boldsymbol{\alpha}:=(\alpha_1 ,\ldots, \alpha_n )\in
\mathbb{R}^n$ with
$\alpha_k \notin \pi \mathbb{Z}$ for any $k\in \{1,\ldots,n \}$. Using Definition \ref{def-fts}, Definition \ref{def-ffts},   (\ref{eq-1.1}), and properties
of the Fourier transform, we can easily conclude that, for any     $f \in \mathscr{S}(\mathbb{R}^n)$ and $\boldsymbol{u}\in \mathbb{R}^n$, (\ref{eq-1.1})   also holds true.

\begin{definition} {\rm (see \cite{g12014})}
Let $f,g \in L^1(\mathbb{R}^n)$.   The \emph{convolution}
$f*g$ is defined by  setting, for any $\boldsymbol{x}\in \mathbb{R}^n$,
$$(f*g)(\boldsymbol{x}):=\int_{\mathbb{R}^n}f(\boldsymbol{y})
g(\boldsymbol{x}-\boldsymbol{y})\,dy.$$
\end{definition}

 \begin{lemma}\label{lem-2-1}{\rm (see \cite{g12014})}
For any  $u\in\mathscr{S}'(\mathbb{R}^n)$ and $f \in
\mathscr{S}(\mathbb{R}^n)$,
$\mathcal{F}(f*u)=\widehat{f}\ \widehat{u}.$
\end{lemma}

\begin{lemma}\label{lem-2-2}{\rm (see \cite{s1993})}
For any $\beta\in(0,n)$ and $\xi\in \mathbb{R}^n$, the identity $$\mathcal{F}\left(\frac{1}{|\boldsymbol{x}
|^{n-\beta}}\right)(\boldsymbol{\xi})=\gamma(\beta)(2\pi)^{-\beta}
|\boldsymbol{\xi}|^{-\beta}$$ holds true in the distribution sense, that is, for any $\varphi \in \mathscr{S}(\mathbb{R}^n)$,
$$\int_{\mathbb{R}^n}\left(\frac{1}{|\boldsymbol{x}|^{n-\beta}}
\right)\widehat{\varphi}(\boldsymbol{x})\,d\boldsymbol{x}=
\gamma(\beta)\int_{\mathbb{R}^n}(2\pi)^{-\beta}
|\boldsymbol{x}|^{-\beta}{\varphi}(\boldsymbol{x})
\,d\boldsymbol{x},$$
where   $\gamma(\beta)$ is the same as in   \eqref{eq-beta}.
\end{lemma}

\begin{theorem}\label{thm-frp}
For any $\beta\in(0,n)$, $\boldsymbol{\alpha}:=(\alpha_1 ,\ldots, \alpha_n )\in
\mathbb{R}^n$ with
$\alpha_k \notin \pi \mathbb{Z}$ for any $k\in \{1,\ldots,n \}$,    $f\in \mathscr{S}(\mathbb{R}^n)$,   and    $\boldsymbol{u}:=(u_1,\ldots,u_n)\in
\mathbb{R}^n$,
$$\mathcal{F}_{\boldsymbol{\alpha}}\left(I_\beta^{\boldsymbol{\alpha}}
f\right)(\boldsymbol{u})=(2\pi)^{-\beta}
|\widetilde{\boldsymbol{u}}|^{-\beta}\mathcal{F}_{\boldsymbol{\alpha}}
\left( f\right)(\boldsymbol{u})\ in\ \mathscr{S}'(\mathbb{R}^n),$$
where
 $\widetilde{\boldsymbol{u}}:=(\widetilde{u}_1,\ldots,\widetilde{u}
_n )=(u_1\csc(\alpha_1),\ldots,u_n\csc(\alpha_n)) $.  The
$(2\pi)^{-\beta}|\widetilde{\boldsymbol{u}}|^{-\beta}$ is called the   \emph{symbol}  of the fractional Riesz potential related to chirp functions.
\end{theorem}

\begin{proof}
Fix an $f\in \mathscr{S}(\mathbb{R}^n)$. From   \eqref{eq-1.1} and Definition \ref{def-frp}, we infer that, for any $\boldsymbol{u} \in
\mathbb{R}^n$,
\begin{align*}
\mathcal{F}_{\boldsymbol{\alpha}}\left(I_\beta^{\boldsymbol{\alpha}}
f\right)(\boldsymbol{u})=&\ c(\boldsymbol{\alpha} )
e_{\boldsymbol{\alpha}}(\boldsymbol{u})\mathcal{F}
(e_{\boldsymbol{\alpha}} I_\beta^{\boldsymbol{\alpha}}f)
(\widetilde{\boldsymbol{u}}) \\
=&\  \frac{1}{\gamma(\beta)}c(\boldsymbol{\alpha} )
e_{\boldsymbol{\alpha}}(\boldsymbol{u})\mathcal{F}
\left(e_{\boldsymbol{\alpha}}f*\left(\frac{1}{|\boldsymbol{\cdot}|
^{n-\beta}}\right)\right )(\widetilde{\boldsymbol{u}}).
\end{align*}
By Lemma \ref{lem-2-1} and Lemma \ref{lem-2-2}, we obtain,  for any $\boldsymbol{u} \in
\mathbb{R}^n$,
\begin{align*}
\mathcal{F}_{\boldsymbol{\alpha}}\left(I_\beta^{\boldsymbol{\alpha}}
f\right)(\boldsymbol{u})=&\
\frac{1}{\gamma(\beta)}c(\boldsymbol{\alpha} )e_{\boldsymbol{\alpha}}
(\boldsymbol{u})\mathcal{F}(e_{\boldsymbol{\alpha}}f)
(\widetilde{\boldsymbol{u}})
\mathcal{F} \left(\frac{1}{|\boldsymbol{\cdot}|^{n-\beta}}\right)
(\widetilde{\boldsymbol{u}})\\
=&\ \frac{1}{\gamma(\beta)}c(\boldsymbol{\alpha} )
e_{\boldsymbol{\alpha}}(\boldsymbol{u})\mathcal{F}
(e_{\boldsymbol{\alpha}}f)(\widetilde{\boldsymbol{u}})
\gamma(\beta)(2\pi)^{-\beta}
|\widetilde{\boldsymbol{u}}|^{-\beta} \\
=&\ (2\pi)^{-\beta}|\widetilde{\boldsymbol{u}}|^{-\beta}\mathcal{F}
_{\boldsymbol{\alpha}}\left( f\right)(\boldsymbol{u}),
\end{align*}
which completes the proof of   Theorem \ref{thm-frp}.
\end{proof}

\begin{remark}
According to Theorem \ref{thm-frp}, we can easily observe
that the fractional Riesz potential related to chirp functions has the semigroup
property, that is, for any $\boldsymbol{\alpha}:=(\alpha_1, \ldots,
\alpha_n)\in  \mathbb{R}^n$ with
$\alpha_k \notin \pi \mathbb{Z}$ for any $k\in\{1,\ldots,n\}$, $\beta_{1},\beta_{2}\in(0,n)$, and $\beta_{1}+\beta_{2}\in(0,n)$,
$I_{\beta_1}^{\boldsymbol{\alpha}}I_{\beta_2}^{\boldsymbol
{\alpha}}=I_{\beta_{1}+\beta_{2}}^{\boldsymbol{\alpha}}.$
\end{remark}

\begin{lemma}\label{lem-ifrft}{\rm(FRFT inversion theorem)
(see \cite{kr2020})}
For any $\boldsymbol{\alpha}:=(\alpha_1, \ldots,
\alpha_n)\in  \mathbb{R}^n$, $f\in \mathscr{S}(\mathbb{R}^n)$,  and   $\boldsymbol{x}\in\mathbb{R}^n$,
$$f(\boldsymbol{x})=\int_{\mathbb{R}^n}\mathcal{F}_
{\boldsymbol{\alpha}}(f)(\boldsymbol{u})K_{-\boldsymbol
{\alpha}}(\boldsymbol{u},\boldsymbol{x})\,d\boldsymbol{u},$$
where $K_{-\boldsymbol{\alpha}}$ is the same as in Definition \ref{def-fft} with $\boldsymbol{\alpha}$ replace by $-\boldsymbol{\alpha}$.
\end{lemma}

From Definition \ref{def-ffts} and Lemma \ref{lem-ifrft}, it follows that, for any $\boldsymbol{\alpha}:=(\alpha_1, \ldots,
\alpha_n)\in  \mathbb{R}^n$ with
$\alpha_k \notin \pi \mathbb{Z}$ for any $k\in\{1,\ldots,n\}$, and for any $f\in \mathscr{S}'(\mathbb{R}^n)$, $\mathcal{F}_{\boldsymbol{-\alpha}}\mathcal{F}_{\boldsymbol{\alpha}}f=f$. By Theorem \ref{thm-frp} and Lemma \ref{lem-ifrft},  we conclude that,
  for any
$\boldsymbol{\alpha}:=(\alpha_1, \ldots,
\alpha_n)\in  \mathbb{R}^n$ with
$\alpha_k \notin \pi \mathbb{Z}$ for any $k\in\{1,\ldots,n\}$, $\beta\in(0,n)$,  $f\in \mathscr{S}(\mathbb{R}^n)$, and $\boldsymbol{x}\in\mathbb{R}^n$,  the fractional Riesz potential related to chirp functions  can be rewritten as    \[
(I_\beta^{\boldsymbol{\alpha}}f) (\boldsymbol{x}) =
\mathcal{F}_{-\boldsymbol{\alpha}}\left( (2\pi)^{-\beta}
|\widetilde{\boldsymbol{u}}|^{-\beta}
(\mathcal{F}_{\boldsymbol{\alpha}}f)
(\boldsymbol{u}) \right)  (\boldsymbol{x})
\]
in  $\mathscr{S}'(\mathbb{R}^n)$, where $\boldsymbol{u}\in\mathbb{R}^n$  and $\widetilde{\boldsymbol{u}}$ is the same as in Theorem \ref{thm-frp}. In what follows, let $$m_\beta^{\boldsymbol{\alpha}} ({\boldsymbol{u}})
:= (2\pi)^{-\beta} |\widetilde{\boldsymbol{u}}|^{-\beta},$$
which is called  the \emph{symbol} of $I_\beta^{\boldsymbol{\alpha}}$.

It is easy to show that the fractional Riesz potential related to chirp functions, $I_\beta^{\boldsymbol{\alpha}}f$, of any $f\in \mathscr{S}(\mathbb{R}^n)$ can be
decomposed into the composition of the multidimensional  FRFT of order $\boldsymbol{\alpha}$, the symbol $m_\beta^{\boldsymbol{\alpha}}({\boldsymbol{u}})$ of the
fractional Riesz potential  related to chirp functions, and the   multidimensional  FRFT of order $\boldsymbol{-\alpha}$,   as shown in the following Fig.
\ref{FIG1}:

\begin{enumerate}
[(i)]
	
\item multidimensional  FRFT of order $\boldsymbol{\alpha}$, namely $g^{\boldsymbol{\alpha}}
({\boldsymbol{u}}):=(\mathcal{F}_{\boldsymbol{\alpha}} f)
({\boldsymbol{u}})$ for any $\boldsymbol{u} \in
\mathbb{R}^n$;
	
\item multiplication by the symbol of the fractional Riesz
potential related to chirp functions $m_\beta^{\boldsymbol{\alpha}}({\boldsymbol{u}})$, namely
$$h_\beta^{\boldsymbol{\alpha}}({\boldsymbol{u}})
:= m_\beta^{\boldsymbol{\alpha}}({\boldsymbol{u}})
g^{\boldsymbol{\alpha}}({\boldsymbol{u}})\  \mathrm{for}\  \mathrm{any}\  \boldsymbol{u} \in
\mathbb{R}^n;$$
	
\item multidimensional  FRFT of order $-\boldsymbol{\alpha}$, namely
$(I_\beta^{\boldsymbol{\alpha}}f) (\boldsymbol{x}):=
(\mathcal{F}_{-\boldsymbol{\alpha}} h_\beta^
{\boldsymbol{\alpha}})(\boldsymbol{x})$ for any $\boldsymbol{x} \in
\mathbb{R}^n$.
\end{enumerate}

\begin{figure}[H]
\centering
\begin{tikzpicture}[thick]
\node (start){$f(\boldsymbol{x})$};
\node[node distance=18mm, rectangle,draw,right of=start]
(FRFT){$\mathcal{F}_{\boldsymbol{\alpha}}$};
\node[node distance=18mm, inner sep=0pt,right of=FRFT]
(MA){$\bigotimes$};
\node[node distance=14mm, below of =MA] (p){$m_\beta^{\boldsymbol{\alpha}}
(\boldsymbol{u})$};
\node[node distance=18mm, rectangle,draw,right of=MA]
(iFRFT){$\mathcal{F}_{-\boldsymbol{\alpha}}$};
\node[node distance=24mm, right of=iFRFT] (end)
{$(I_\beta^{\boldsymbol{\alpha}}f)(\boldsymbol{x})$};
\draw[->](start)--(FRFT);
\draw[->](FRFT)--node[above]{$g^{\boldsymbol{\alpha}}
(\boldsymbol{u})$}(MA);
\draw[->](MA)--node[above]{$h_\beta^{\boldsymbol{\alpha}}
(\boldsymbol{u})$}(iFRFT);
\draw[->](iFRFT)--(end);
\draw[->](p)--(MA);
\end{tikzpicture}
\caption{The decomposition of $I_\beta^{\boldsymbol{\alpha}}(f)$}%
\label{FIG1}%
\end{figure}
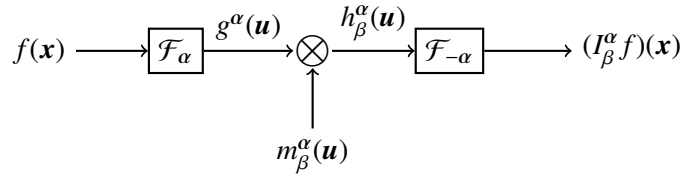

\begin{remark}
The symbol $m_\beta^{\boldsymbol{\alpha}}({\boldsymbol{u}})$ of the fractional Riesz potential related to chirp functions
  does not
contain $i$, and hence  there is no phase-shift  effect. It can only change the amplitude,
and hence the fractional Riesz potential related to chirp functions $m_\beta^{\boldsymbol{\alpha}}({\boldsymbol{u}})$ is  only an amplitude modulator.
\end{remark}

In \cite{fglwy2021}, the derivative formula of the multidimensional  FRFT was established
as follows.

\begin{lemma}\label{lem-frftdf}{\rm (see \cite{fglwy2021})}
{\rm (multidimensional  FRFT derivative formula)}
Let $\boldsymbol{\alpha}:=(\alpha_1, \ldots,
\alpha_n)\in  \mathbb{R}^n$ with
$\alpha_k \notin \pi \mathbb{Z}$ for any $k\in\{1,\ldots,n\}$, $f\in L^1(\mathbb{R}^n)$, and $e_{\boldsymbol{\alpha}}$ be the same as in Remark  \ref{rem-fft-ft}. If $k\in\{1,\ldots,n\}$ and $e_{\boldsymbol{ \alpha}}f$ is
absolutely continuous on $\mathbb{R}^n$ with respect to the $k$th
variable, then, for any $\boldsymbol{x}=(x_1, \ldots,x_n)\in \mathbb{R}^n$,
$$\mathcal{F}_{\boldsymbol{\alpha}}\left(e_{-\boldsymbol{\alpha}}
(\boldsymbol{y})\frac{\partial[e_{ \boldsymbol{\alpha}}
(\boldsymbol{y})f(\boldsymbol{y})]}{\partial y_k}\right)
(\boldsymbol{x})=2\pi ix_k\csc (\alpha_k)\mathcal{F}_\alpha(f)
(\boldsymbol{x}).$$
\end{lemma}

Then the relations among the fractional Riesz potential  related to chirp functions, the multidimensional  FRFT, and the fractional Riesz transform  related to chirp functions can be established as follows.

\begin{theorem} \label{thm-frt_frp}

For any $\boldsymbol{\alpha}:=(\alpha_1, \ldots,
\alpha_n)\in  \mathbb{R}^n$ with
$\alpha_k \notin \pi \mathbb{Z}$ for any $k\in\{1,\ldots,n\}$, and for $f\in \mathscr{S}(\mathbb{R}^n)$, one has
\begin{align*}
f =I_1^{\boldsymbol{\alpha}}\left(\sum^n_{j=1}
R_j^{\boldsymbol{\alpha}}\left(e_{-\boldsymbol{\alpha}}
\partial_j(e_{ \boldsymbol{\alpha}} f)\right) \right)
 =\sum^n_{j=1} I_1^{\boldsymbol{\alpha}}
\left( R_j^{\boldsymbol{\alpha}}\left(e_{-\boldsymbol{\alpha}}
\partial_j(e_{ \boldsymbol{\alpha}}f)\right)  \right)
 \end{align*}
in  $\mathscr{S}'(\mathbb{R}^n)$,
where $e_{\boldsymbol{\alpha}}$ is the same as in Remark  \ref{rem-fft-ft}, $R_j^{\boldsymbol{\alpha}}$ for any $j\in \{1,\ldots,n\}$ is the  jth fractional Riesz transform, and  $I_1^{\boldsymbol{\alpha}}$ is the
fractional Riesz potential $I_\beta^{\boldsymbol{\alpha}}$ with $\beta=1$.
\end{theorem}

\begin{proof}From   Theorem \ref{thm-frp}, \eqref{eq-R}, and
Lemma \ref{lem-frftdf}, we   deduce that, for any $\boldsymbol{x}\in \mathbb{R}^n$,
\begin{align*}
&\mathcal{F}_{\boldsymbol{\alpha}}\left(I_1^{\boldsymbol{\alpha}}
\left(\sum^n_{j=1}R_j^{\boldsymbol{\alpha}}\left(
e_{-\boldsymbol{\alpha}}  \partial_j(e_{ \boldsymbol{\alpha}}f)
\right) \right) \right)(\boldsymbol{x})\\
&\quad=(2\pi)^{-1}|\widetilde{
\boldsymbol{x}}|^{-1}\  \mathcal{F}_{\boldsymbol{\alpha}}\left(
\sum^n_{j=1}R_j^{\boldsymbol{\alpha}}\left(e_{-\boldsymbol{\alpha}}
\partial_j(e_{ \boldsymbol{\alpha}}f)\right) \right)
(\boldsymbol{x})\\
&\quad=(2\pi)^{-1}|\widetilde{\boldsymbol{x}}|^{-1}\sum^n_{j=1}-i
\frac{\widetilde{x}_j}{|\widetilde{\boldsymbol{x}}|}\mathcal{F}_{\boldsymbol{
\alpha}}\left(e_{-\boldsymbol{\alpha}}  \partial_j(
e_{ \boldsymbol{\alpha}}f) \right) (\boldsymbol{x})\\
&\quad=(2\pi)^{-1}|\widetilde{\boldsymbol{x}}|^{-1}\sum^n_{j=1}-i
\frac{\widetilde{x}_j}{|\widetilde{\boldsymbol{x}}|}2\pi i\widetilde{x}_j
\mathcal{F}_{\boldsymbol{\alpha}}(f)(\boldsymbol{x})=\mathcal{F}_{\boldsymbol{\alpha}}(f)(\boldsymbol{x}).
\end{align*}
Then, by taking the multidimensional  FRFT of order $-\boldsymbol{\alpha}$ and using Lemma \ref{lem-ifrft}, we obtain the desired identity, which
completes the proof of   Theorem \ref{thm-frt_frp}.
\end{proof}

\begin{remark}
When $\boldsymbol{\alpha}=(\frac{\pi}{2}+k_1\pi, \ldots,
\frac{\pi}{2}+k_n\pi)$ with  $\{k_j\}_{j=1}^n\subset \mathbb{Z}$, in this case,
Theorem \ref{thm-frp} goes back to \eqref{eq-RP} of \cite[pp. 124]{ldy2007} and
Theorem \ref{thm-frt_frp} goes back to  (3.0.3) of
\cite[pp. 124]{ldy2007}
 .
\end{remark}

From Lemma \ref{lem-frftdf}, we can easily  deduce the following
theorem; we omit the details.
\begin{theorem}\label{thm-frftdf}
Let   $\boldsymbol{\alpha}:=(\alpha_1, \ldots,
\alpha_n)\in  \mathbb{R}^n$ with
$\alpha_k \notin \pi \mathbb{Z}$ for any $k\in\{1,\ldots,n\}$,   $\boldsymbol{\varsigma}\in\zz_+^n$, and  $e_{\boldsymbol{\alpha}}$ be the chirp function. For any $f\in \mathscr{S}(\mathbb{R}^n)$ and
$\boldsymbol{x}:=(x_1,\ldots,x_n)\in \mathbb{R}^n$,  one has
\begin{align*}
\mathcal{F}_{\boldsymbol{\alpha}}\left(e_{-\boldsymbol{\alpha}}
(\boldsymbol{y}) {\partial^{\boldsymbol{\varsigma}}
[e_{ \boldsymbol{\alpha}}(\boldsymbol{y})f(\boldsymbol{y})]}
\right)(\boldsymbol{x})=(2\pi i\widetilde{\boldsymbol{x}})^
{\boldsymbol{\varsigma}}
\mathcal{F}_\alpha (f)(\boldsymbol{x}),
\end{align*}
where $\widetilde{\boldsymbol{x}}:=(\widetilde{x}_1 ,\ldots,\widetilde{x}
_n ):=(x_1\csc(\alpha_1),\ldots,x_n\csc(\alpha_n))$.
\end{theorem}

 From Lemma \ref{lem-frftdf},   it follows that, for any $\boldsymbol{\alpha}:=(\alpha_1, \ldots,
\alpha_n)\in  \mathbb{R}^n$ with
$\alpha_k \notin \pi \mathbb{Z}$ for any $k\in\{1,\ldots,n\}$, $f\in \mathscr{S}(\mathbb{R}^n)$, and  $\boldsymbol{x}:=(x_1,\ldots,x_n)\in \mathbb{R}^n$,
\begin{align}\label{eq-la}
\mathcal{F}_{\boldsymbol{\alpha}}\left(-e_{-\boldsymbol{\alpha}}
 { \Delta[e_{ \boldsymbol{\alpha}}
 f ]}\right)(\boldsymbol{x})
=4\pi^2|\widetilde{\boldsymbol{x}}|^{2}\mathcal{F}_\alpha (f)
(\boldsymbol{x}),
\end{align}
where $\widetilde{\boldsymbol{x}}:=(\widetilde{x}_1 ,\ldots,\widetilde{x}
_n ):=(x_1\csc(\alpha_1),\ldots,x_n\csc(\alpha_n))$, $e_{\boldsymbol{\alpha}}$ is the same as in Remark  \ref{rem-fft-ft}, and $\Delta$ is the \emph{Laplace operator}
$ \sum_{j=1}^n\frac{\partial^{2}}{\partial  x_j^2}.$

Motivated by \eqref{eq-la}, we replace the exponent $2$ in \eqref{eq-la} by a
positive real number to introduce  fractional Laplace operators
related to chirp functions.

\begin{definition} \label{def-frftd}
Suppose $\boldsymbol{\alpha}:=(\alpha_1, \ldots,
\alpha_n)\in  \mathbb{R}^n$ with
$\alpha_k \notin \pi \mathbb{Z}$ for any $k\in\{1,\ldots,n\}$, $e_{\boldsymbol{\alpha}}$ is the same as in Remark  \ref{rem-fft-ft}, $z\in \mathbb{R}$, and  $z\in(0,\infty)$. The
\emph{fractional Laplace operator related to chirp functions}, $\Delta_{z}$, is defined by setting, for any $f\in \mathscr{S}(\mathbb{R}^n)$   and $\boldsymbol{x}:=(x_1,\ldots,x_n) \in \mathbb{R}^n$,
$$\mathcal{F}_{\boldsymbol{\alpha}}\left(\left[-e_{-\boldsymbol
{\alpha}}(\boldsymbol{y}) { \Delta_{z}[e_{ \boldsymbol{\alpha}}
(\boldsymbol{y})f(\boldsymbol{y})]}\right]^{\frac{z}{2}}\right)
(\boldsymbol{x})
:=(2\pi)^z|\widetilde{\boldsymbol{x}}|^{z}\mathcal{F}_\alpha(f)
(\boldsymbol{x}),$$
where    $\widetilde{\boldsymbol{x}}:=(\widetilde{x}_1 ,\ldots,\widetilde{x}
_n)=(x_1\csc(\alpha_1),\ldots,x_n\csc(\alpha_n))$.
The
$(2\pi)^z|\widetilde{\boldsymbol{x}}|^{z}$ is called the   \emph{symbol} of the
fractional Laplace operators related to chirp functions, $\Delta_{z}$.
\end{definition}

Obviously, the fractional Laplace operator related to chirp functions, $\Delta_{z}$, can be
viewed as the inverse operator of the fractional Riesz potential related to chirp functions and plays a key role in the decryption process
of the image encryption   in Section \ref{sec4}.

\subsection{Boundedness of fractional Riesz potentials related to chirp functions} \label{sec2.2}
In this subsection, we establish the boundedness of the fractional Riesz
potential related to chirp functions on the rotation invariant space related to chirp functions, which
is equivalent to the boundedness of the classical Riesz
potential  on the same rotation invariant space related to chirp functions.

\begin{definition}\label{def-ris}
Let $(X,\|\cdot\|_X)$ be a Banach space. Then   $X$ is called
the \emph{rotation invariant space related to chirp functions} if, for any $\boldsymbol{\alpha}:=(\alpha_1, \ldots,
\alpha_n)\in  \mathbb{R}^n$ with
$\alpha_k \notin \pi \mathbb{Z}$ for any $k\in\{1,\ldots,n\}$, and for any $f\in X$,
$$\|e_{\boldsymbol{\alpha}}f\|_X=\|f\|_X,$$
  where $e_{\boldsymbol{\alpha}}$ is the same as in Remark  \ref{rem-fft-ft}.
\end{definition}

\begin{remark}
The well-konwn rotation invariant space  includes the Lebesgue space,
the Morrey space, and the Herz space.
\end{remark}

\begin{theorem}\label{the-rb}
If $X, Y$ are two rotation invariant spaces related to chirp functions, $\beta\in(0,n)$, and
$\boldsymbol{\alpha}:=(\alpha_1, \ldots,
\alpha_n)\in  \mathbb{R}^n$ with
$\alpha_k \notin \pi \mathbb{Z}$ for any $k\in\{1,\ldots,n\}$, then $I_\beta$ is
bounded from $X$ to $Y$ if and only if
$I_\beta^{\boldsymbol{\alpha}}$ is bounded from $X$ to $Y$.
\end{theorem}

\begin{proof} Let $\|I_\beta\|_{X\rightarrow Y}
<\infty$. By Definition \ref{def-frp} and Definition \ref{def-ris}, we conclude  that,   for any  $f\in X$,
\begin{align*}
\|I_\beta^{\boldsymbol{\alpha}}(f)\|_{Y}=\|I_\beta
(e_{\boldsymbol{\alpha}}f) \|_{Y}
\ls\|e_{\boldsymbol{\alpha}}f\|_{X}
= \|f\|_{X}.
\end{align*}
Conversely, if
$\|I_\beta^{\boldsymbol{\alpha}} \|_{X\rightarrow Y}<\infty$,
then we obtain,   for any  $f\in X$,
\begin{align*}
\|I_\beta(f)\|_{Y}=\|e_{-\boldsymbol{\alpha}}I_\beta(f)\|_{Y}
=\|I_\beta^{\boldsymbol{\alpha}}(e_{-\boldsymbol{\alpha}}f)\|_{Y}
\ls\|e_{-\boldsymbol{\alpha}}f\|_{X}
=\|f\|_{X}.
\end{align*}
This finishes the proof of Theorem  \ref{the-rb}.
\end{proof}

From Theorem \ref{the-rb}, we can easily  deduce the following
corollary; we omit the details.

\begin{corollary}{\rm (Hardy--Littlewood--Sobolev theorem)}
\label{cor-bound}
Let $\beta\in(0,n)$,   $\boldsymbol{\alpha}:=(\alpha_1, \ldots,
\alpha_n)\in  \mathbb{R}^n$ with
$\alpha_k \notin \pi \mathbb{Z}$ for any $k\in\{1,\ldots,n\}$, $p\in[1,\frac{n}{\beta})$,
and  $\frac{1}{q}:=\frac{1}{p}-\frac{n}{\beta}.$
	\begin{itemize}
		\item[\rm (i)] If $p\in(1,\frac{n}{\beta})$,
then there exists a positive constant $C$ such that, for any  $f\in L^p(\mathbb{R}^n)$, $$\|I_\beta^{\boldsymbol{\alpha}}f\|_{L^q(\mathbb{R}^n)}\leq C \|f\|_{L^p(\mathbb{R}^n)};$$
  \item[\rm (ii)]    There exists a positive constant $C$ such that, for     $f\in L^1(\mathbb{R}^n)$ and any
$\lambda\in(0,\infty)$, $$|\{x\in \mathbb{R}^n:|I_\beta^{\boldsymbol{\alpha}}
f({\boldsymbol{x}})|>\lambda\}|\le C \left[\frac{1}{\lambda}\|f\|
_{L^1(\mathbb{R}^n)}\right]^{\frac{n}{n-\beta}}.$$
	\end{itemize}
\end{corollary}

\section{Simulation of   fractional Riesz
potentials related to chirp functions\label{sec3}}

In this section, we apply the fractional Riesz potential related to chirp functions to
an image with the help of the FRFT discrete
algorithm (see, for instance, \cite{bm2004,bm,tlz2010}).

 In the following Fig. \ref{FIG2}, we  give  the numerical simulation of $I_{\beta}^{(\pi/4,\pi/4)}$. In the continuous case, the following Fig.
\ref{FIG2}(a)  can be regarded as  the function
\[f(x_1,x_2):=
   \left\{
\begin{array}[c]{ll}
0,  & \quad  \forall\,(x_1,x_2) \in [0,200]^2 \cup [200,400]^2,  \\
255,  & \quad {\rm otherwise}
		\end{array}
   \right.
\]
on $\mathbb{R}^2$.

\begin{figure}[H]
\centering
\subfigure[]{\includegraphics[width=0.3\linewidth]{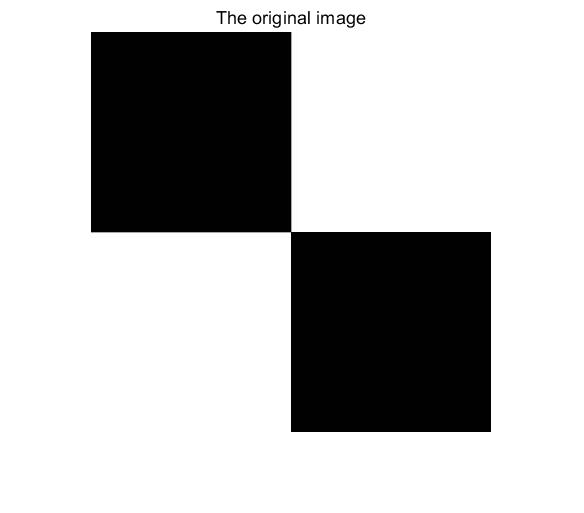} }
\subfigure[]{\includegraphics[width=0.3\linewidth]{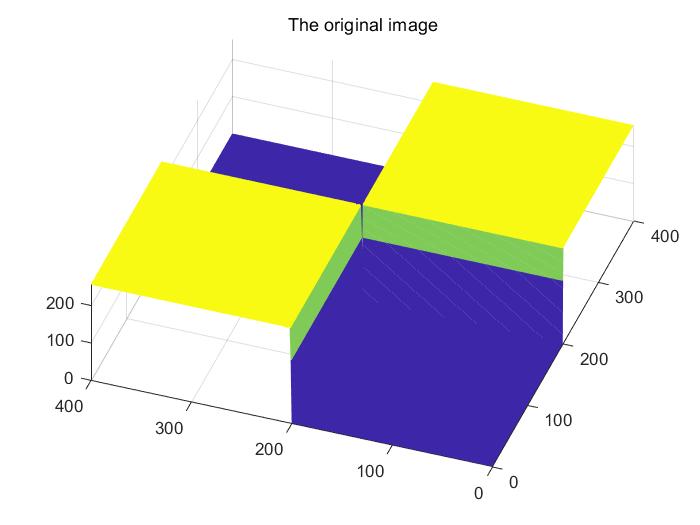} }
\subfigure[]{\includegraphics[width=0.3\linewidth]{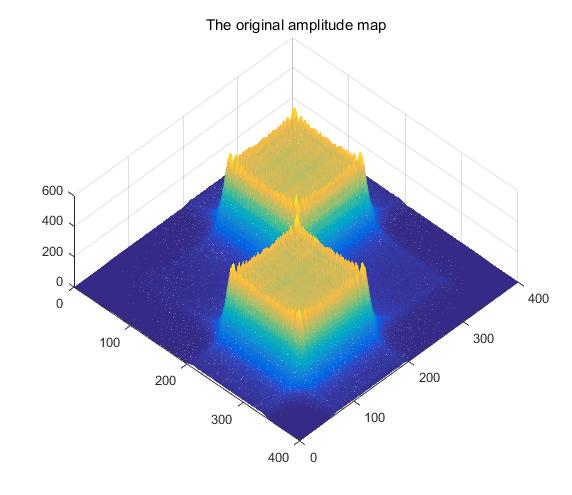} }
\subfigure[]{\includegraphics[width=0.3\linewidth]{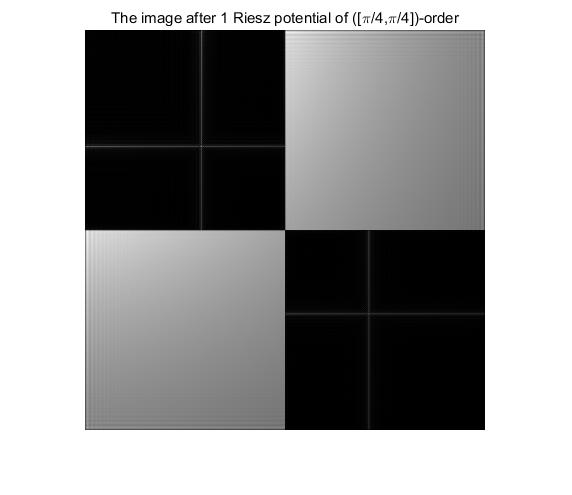} }
\subfigure[]{\includegraphics[width=0.3\linewidth]{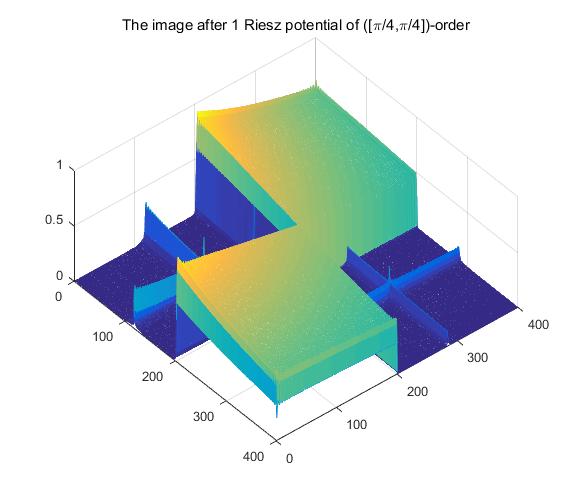} }
\subfigure[]{\includegraphics[width=0.3\linewidth]{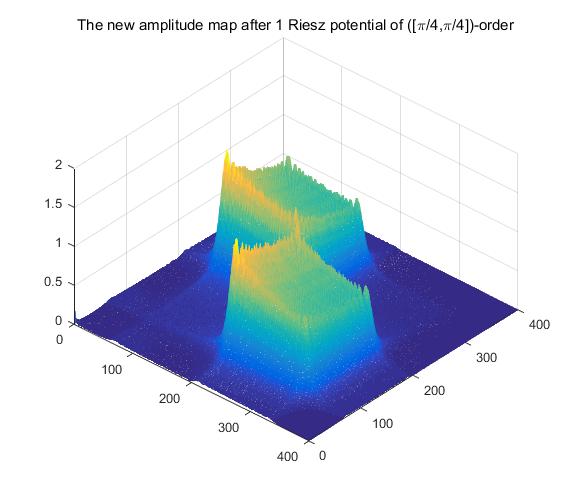} }
\end{figure}

\begin{figure}[H]
\centering
\subfigure[]{\includegraphics[width=0.3\linewidth]{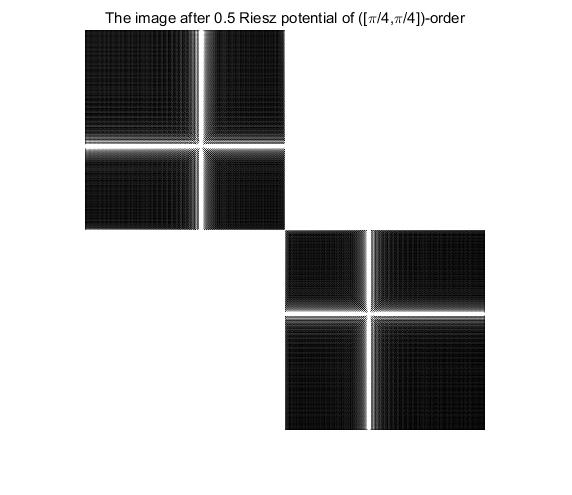} }
\subfigure[]{\includegraphics[width=0.3\linewidth]{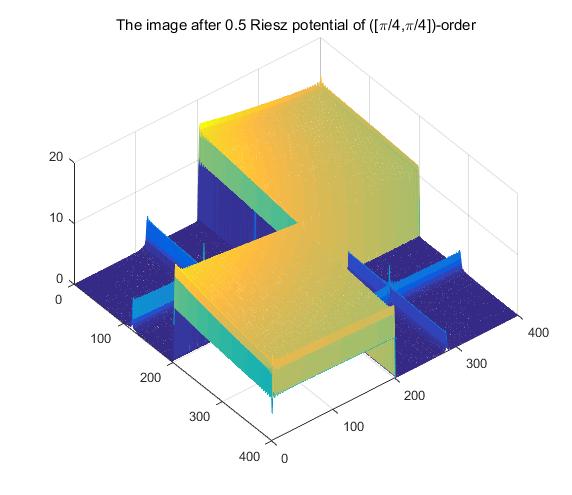} }
\subfigure[]{\includegraphics[width=0.3\linewidth]{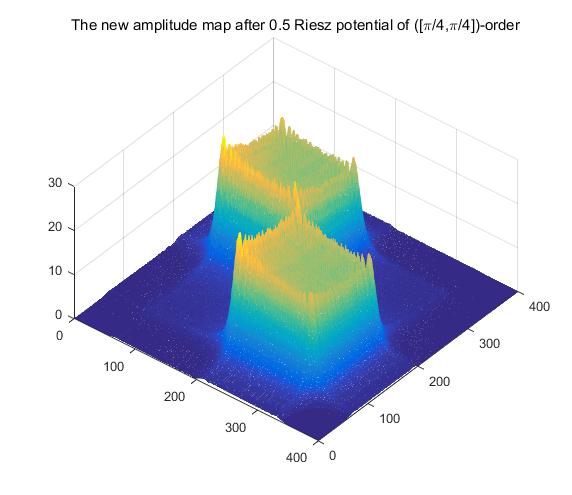} }
\subfigure[]{\includegraphics[width=0.3\linewidth]{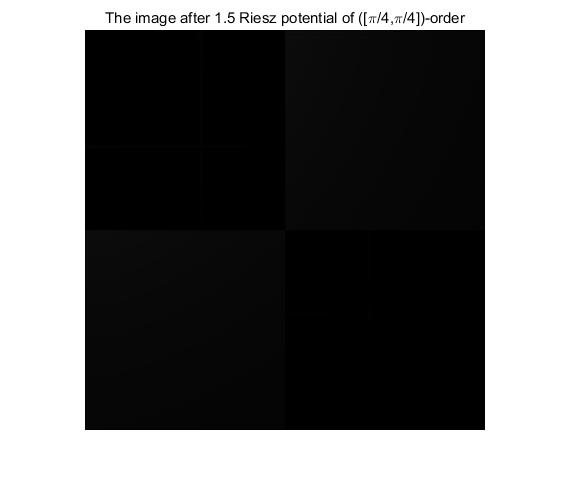} }
\subfigure[]{\includegraphics[width=0.3\linewidth]{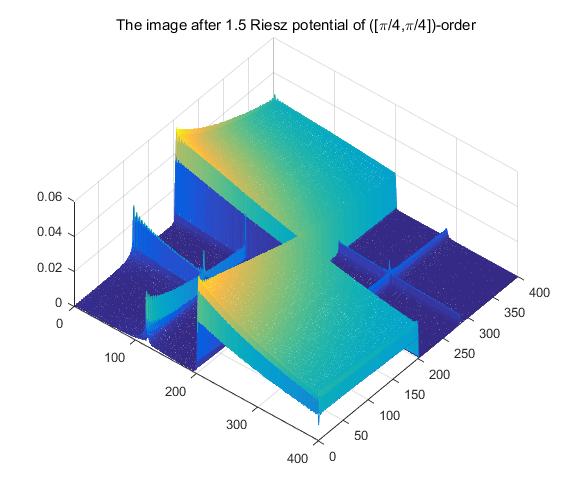} }
\subfigure[]{\includegraphics[width=0.3\linewidth]{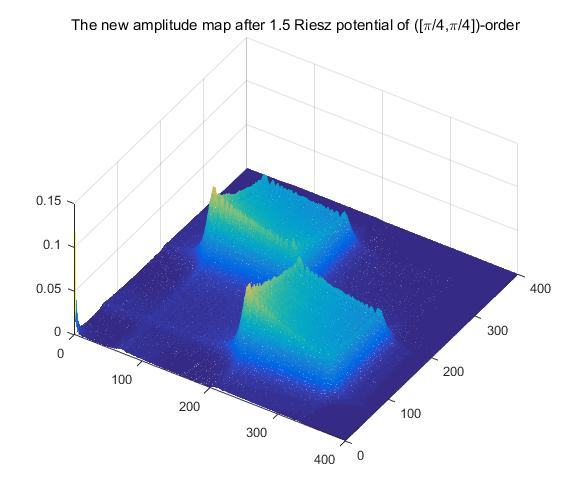} }
\caption{Numerical simulation of $I_{\beta}^{(\pi/4,\pi/4)}$}
\label{FIG2}%
\end{figure}

As is shown in the above Fig.
\ref{FIG2}, (a) is the original 2-dimensional   grayscale
image with 400 pixels $\times$ 400 pixels; (d), (g), and (k)
are the 2-dimensional   grayscale images, respectively, after
$I_{1}^{(\pi/4,\pi/4)}f$, $I_{0.5}^{(\pi/4,\pi/4)}f$, and
$I_{1.5}^{(\pi/4,\pi/4)}f$.    Graphs
(b), (e), (h), and (k) of Fig. \ref{FIG2} are the 3-dimensional color graphs
of $f$, $I_{1}^{(\pi/4,\pi/4)}f$, $I_{0.5}^{(\pi/4,\pi/4)}f$,
and $I_{1.5}^{(\pi/4,\pi/4)}f$, respectively.
 Graphs (c), (f), (i), and (l) of Fig. \ref{FIG2} are the amplitude
images  in the fractional Fourier domain of order
$\boldsymbol{\alpha}=(\pi/4,\pi/4)$ of $f$,
$I_{1}^{(\pi/4,\pi/4)}f$, $I_{0.5}^{(\pi/4,\pi/4)}f$,
and $I_{1.5}^{(\pi/4,\pi/4)}f$, respectively.

 In the following  Fig. \ref{FIG3}, we   give  the numerical simulation of
$I_{\beta}^{(\pi/8,3\pi/8)}$.

\begin{figure}[H]
\centering
\subfigure[]{\includegraphics[width=0.3\linewidth]{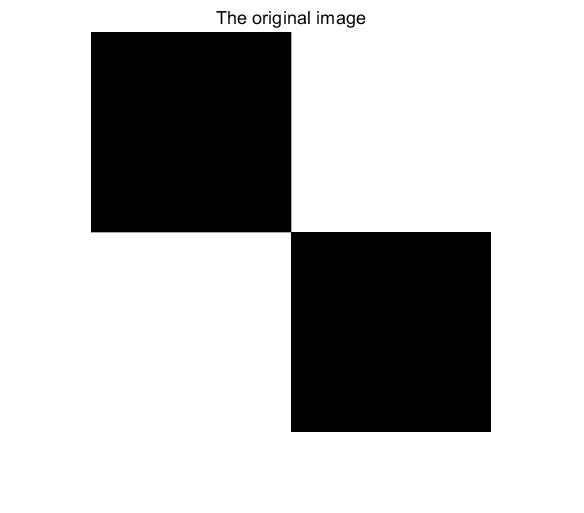} }
\subfigure[]{\includegraphics[width=0.3\linewidth]{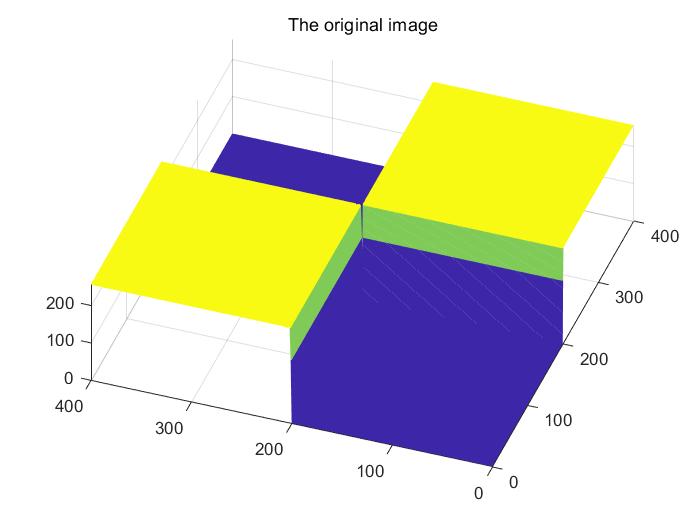} }
\subfigure[]{\includegraphics[width=0.3\linewidth]{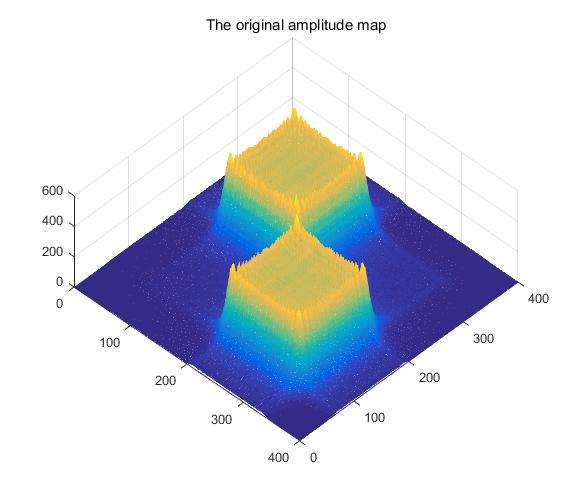} }
\end{figure}

\begin{figure}[H]
\centering
\subfigure[]{\includegraphics[width=0.3\linewidth]{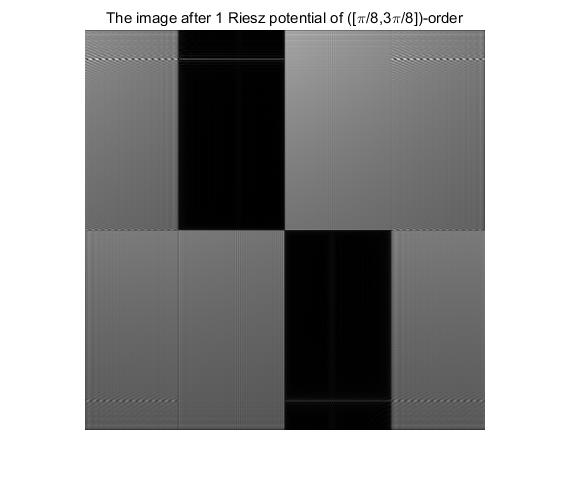} }
\subfigure[]{\includegraphics[width=0.3\linewidth]{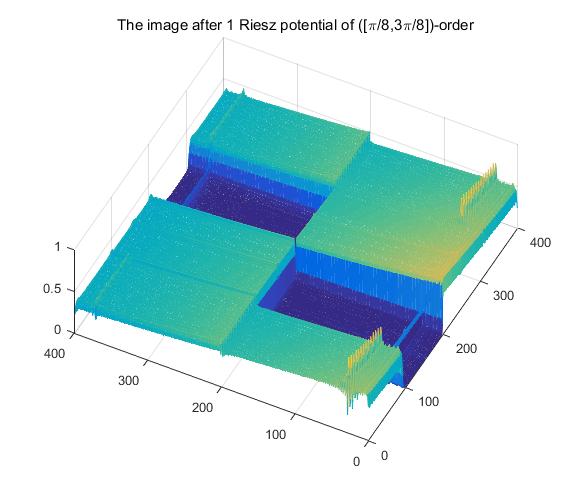} }
\subfigure[]{\includegraphics[width=0.3\linewidth]{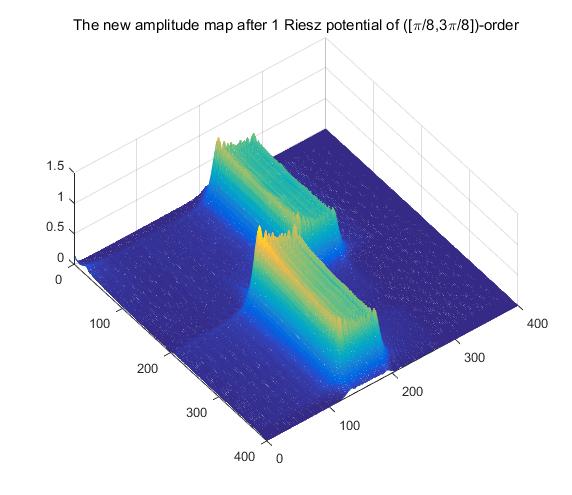} }
\subfigure[]{\includegraphics[width=0.3\linewidth]{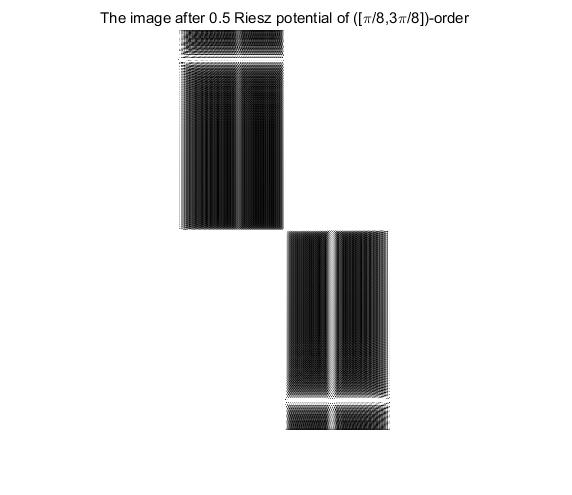} }
\subfigure[]{\includegraphics[width=0.3\linewidth]{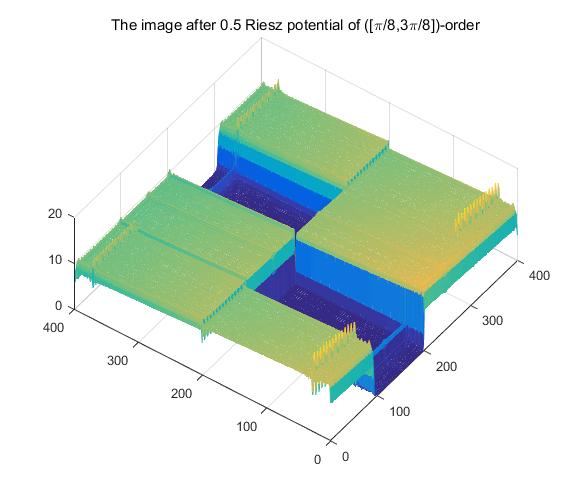} }
\subfigure[]{\includegraphics[width=0.3\linewidth]{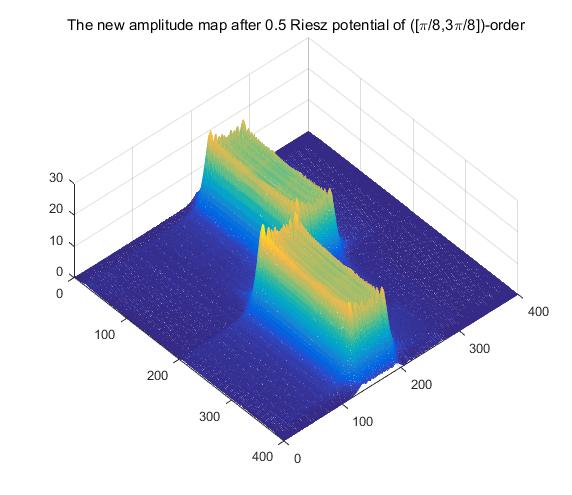} }
\subfigure[]{\includegraphics[width=0.3\linewidth]{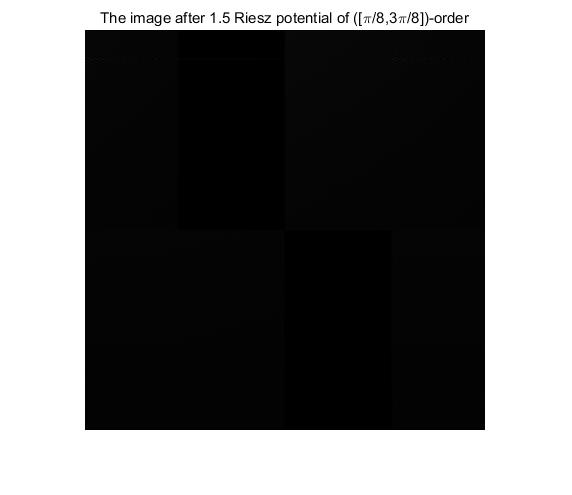} }
\subfigure[]{\includegraphics[width=0.3\linewidth]{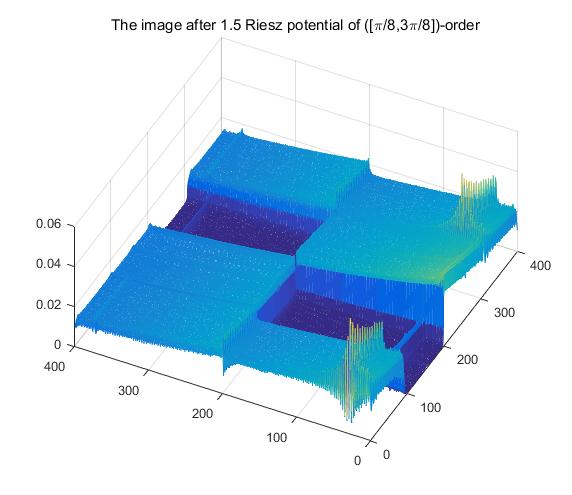} }
\subfigure[]{\includegraphics[width=0.3\linewidth]{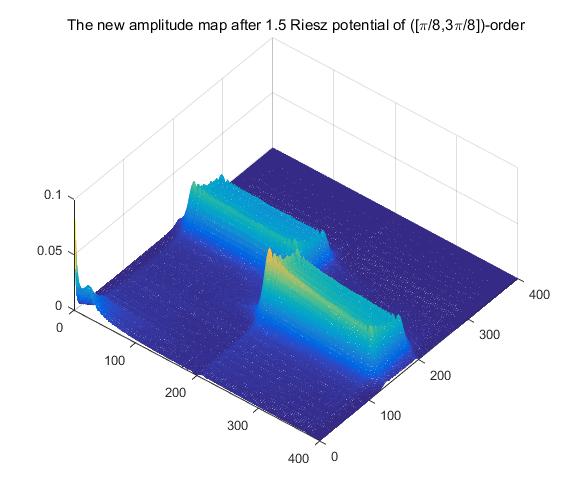} }
\caption{Numerical simulation of $I_{\beta}^{(\pi/8,3\pi/8)}$}
\label{FIG3}%
\end{figure}

As is shown in the above Fig.
\ref{FIG3}, (a) is the original 2-dimensional   grayscale
image with 400 pixels $\times$ 400 pixels; (d), (g), and (k) are
the 2-dimensional   grayscale images, respectively,  after
$I_{1}^{(\pi/8,3\pi/8)}f$, $I_{0.5}^{(\pi/8,3\pi/8)}f$,
and $I_{1.5}^{(\pi/8,3\pi/8)}f$.
   Graphs  (b), (e), (h), and (k) of Fig. \ref{FIG3} are the 3-dimensional
color graphs of $f$, $I_{1}^{(\pi/8,3\pi/8)}f$,
$I_{0.5}^{(\pi/8,3\pi/8)}f$, and $I_{1.5}^{(\pi/8,3\pi/8)}f$.    Graphs  (c), (f), (i), and (l) of Fig. \ref{FIG3} are
the amplitude  images  in the fractional Fourier domain of order
$\boldsymbol{\alpha}=(\pi/8,3\pi/8)$ of
$f$, $I_{1}^{(\pi/8,3\pi/8)}f$,
$I_{0.5}^{(\pi/8,3\pi/8)}f$, and $I_{1.5}^{(\pi/8,3\pi/8)}f$,
respectively.

Comparing     Fig. \ref{FIG2}(i) with     Fig. \ref{FIG3}(i),
    Fig. \ref{FIG2}(f) with     Fig. \ref{FIG3}(f), and
Fig. \ref{FIG2}(l) with   Fig. \ref{FIG3}(l), we find that,  when
 ${\boldsymbol{\alpha}}$ changes  and $\beta$  remains unchanged,
the  image amplitude changes dramatically.
Comparing  (f), (i), and (l) of both Fig. \ref{FIG2} and
Fig. \ref{FIG3}, we conclude that, when  $\beta$   changes and
${\boldsymbol{\alpha}}$ remains unchanged,
the  picture amplitude also changes dramatically.
Graphs  (c), (f), (i), and (l) of both  Fig. \ref{FIG2}
and   Fig. \ref{FIG3}  indicate that the   symbol of the
fractional Riesz potential can correspondingly dramatically  change the amplitude in the
fractional Fourier domain by adjusting parameters
 ${\boldsymbol{\alpha}}$ and $\beta$.
To sum up, Fig. \ref{FIG2} and Fig. \ref{FIG3} show that
the fractional Riesz potential is    an amplitude modulator.
This  is quite different from   the  fractional  multiplier
of the fractional Riesz transform, which is not only a
phase-shift converter but also an amplitude attenuator.

In conclusion, we known that the fractional Hilbert transform  related to chirp functions has
a phase-shift effect in the fractional Fourier domain. However,  the
fractional Riesz transform related to chirp functions not only has a phase-shift in the
fractional Fourier domain, but also can attenuate amplitude.
 Moreover, the fractional Riesz potential  related to chirp functions can change the amplitude
in the fractional Fourier domain. Since these three transforms behave quite different due to  their different
multipliers or symbols, we predict that these three transforms
will play quite different important roles in signal processing.

\section{Image encryption with double phase coding based
on fractional Riesz potentials related to chirp functions}\label{sec4}

With the development of broadband network and multimedia
technology, the acquisition, transmission, and processing
of image data spread to all corners of the digital era.
Security issues are also becoming increasingly serious.
Many image datum need to be transmitted and stored
confidentially, such as photographs taken by satellites,
architectural drawings from financial institutions, and,
in the telemedicine system, patient records and medical images.

In \cite{gbjr1998,ujs2000},   Goudail et al. and Unnikrishnan et al. proposed a double
phase coding image encryption method based on the FRFT.
Compared with the double phase coding image encryption method
based on the Fourier transform in \cite{jszg1997,rj1995},
in addition to the phase mask, the improved double phase
coding encryption key increases the order of the FRFT twice,
and hence expands the key space. When the order is
unknown, it will not be normally decrypted.
Now, we propose a new image encryption method based
on the fractional Riesz potential related to chirp functions with double phase coding.
That is, we   change the amplitude of the FRFT domain
through the symbol of the fractional Riesz potential related to chirp functions,
whose symbol, together with
the order of the FRFT,  provides greater degrees of freedom,
expands the key space, and improves the security of the
protected information.
The following Fig. \ref{FIG4}  exactly  explains this encryption processing.

\begin{figure}[H]
\centering
\includegraphics[width=0.9\linewidth]{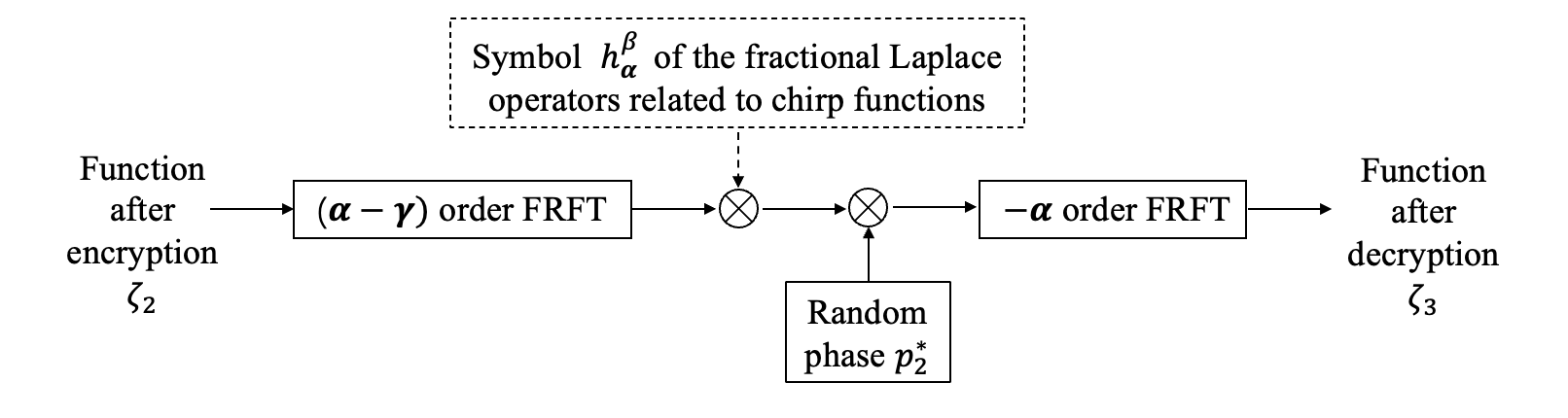}
\caption{Encryption  processing   with double phase coding
based on $I_\beta^{\boldsymbol{\alpha}}$}
\label{FIG4}%
\end{figure}

 As is shown in  Fig. \ref{FIG4},   the input function $\zeta_1(\boldsymbol{x})$ represents
the image to be encrypted after the normalization  and the pixel
value range is $[0,1]$. The function after encryption
$\zeta_2(\boldsymbol{x})$ represents the encryption   image.
The random phases $p_1$ and $p_2$ are given, respectively, by   $p_1(\boldsymbol{x}):=e^{2\pi i n_1(\boldsymbol{x})}$ and
$p_2(\boldsymbol{x}):=e^{2\pi i n_2(\boldsymbol{x})}$ for any $\boldsymbol{x}\in \mathbb{R}^2$, which are
called the random phase masks, where $n_1(\boldsymbol{x})$ and
$n_2(\boldsymbol{x})$  are two statistically independent white
sequences uniformly distributed on $[0,1]$. If the symbol of
the fractional Riesz potential related to chirp functions is not added in the above  encryption processing, this
image encryption  processing goes back to the image encryption with double
phase coding based on FRFT.

Conversely, we next give the following decryption processing  with double
phase coding based on the fractional
Riesz potential related to chirp functions.

\begin{figure}[H]
\centering
\includegraphics[width=0.9\linewidth]{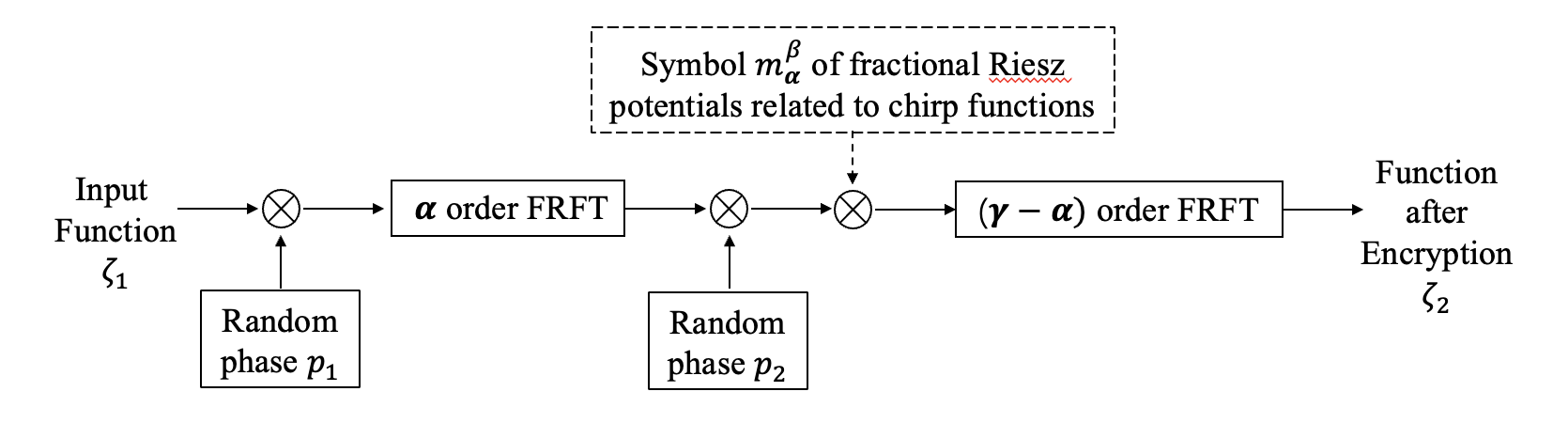}
\caption{Decryption processing  with double
phase coding based on $I_\beta^{\boldsymbol{\alpha}}$}
\label{FIG5}%
\end{figure}

 As is shown in  Fig. \ref{FIG5},
the decryption processing is also the inverse of the encryption
processing. From Section \ref{sec2} we   deduce that the symbol
of the fractional Laplace operator related to chirp functions
multiplied by the symbol of the fractional Riesz potential related to chirp functions
is $1$, that is, the fractional Laplace operator related to chirp
functions can be regarded as the inverse operation of the
fractional Riesz potential related to chirp functions. The random phase $p^*_2$ is defined by setting
$p^*_2(\boldsymbol{x}):=e^{-2\pi i n_2(\boldsymbol{x})}$ for any $\boldsymbol{x}\in \mathbb{R}^2$, which is the
conjugate of the random phase mask $p_2$  in Fig. \ref{FIG4}, and the function after
decryption $\zeta_3(\boldsymbol{x})$ represents   the image
after decryption.

Now,  we study the simulation of both the image encryption and the image decryption.  The following Fig. \ref{FIG6} illustrates the digital simulation of both the
 encryption and the decryption using classical pictures.

\begin{figure}[H]
\centering
\subfigure[]{\includegraphics[width=0.3\linewidth]{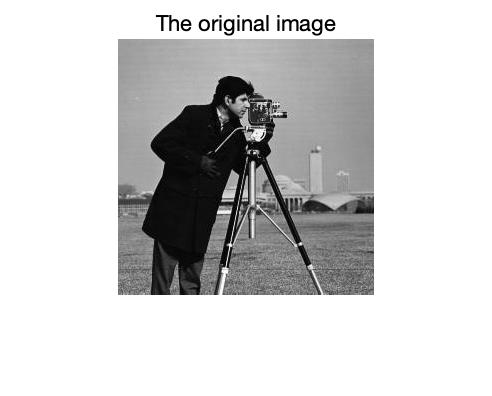} }
\subfigure[]{\includegraphics[width=0.3\linewidth]{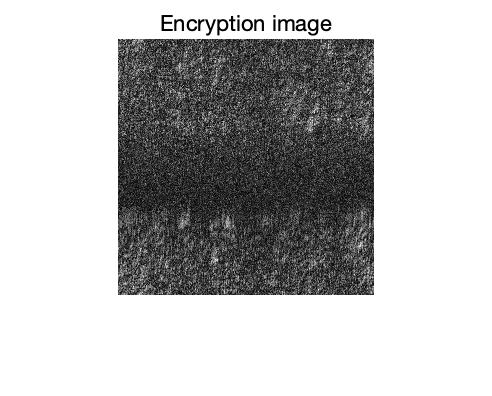} }
\subfigure[]{\includegraphics[width=0.3\linewidth]{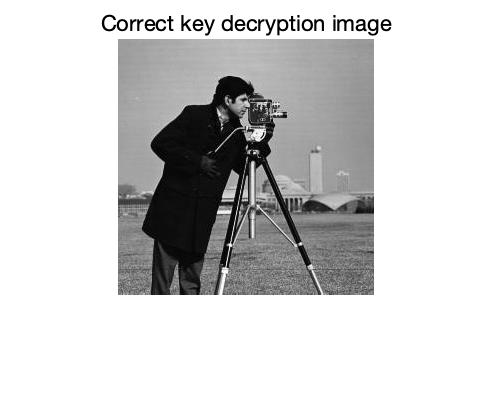} } 	
\end{figure}

\begin{figure}[H]
\centering	
\subfigure[]{\includegraphics[width=0.3\linewidth]{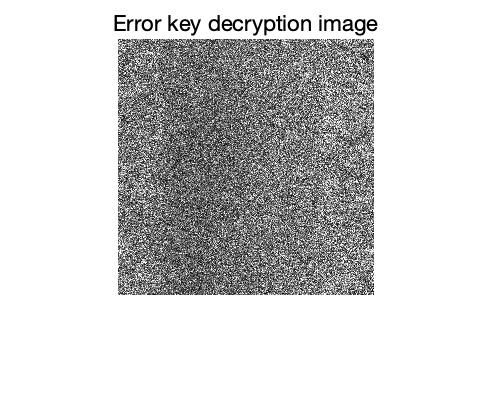} }
\subfigure[]{\includegraphics[width=0.3\linewidth]{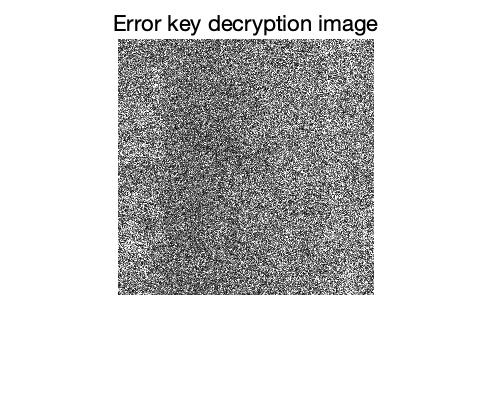} }
\subfigure[]{\includegraphics[width=0.3\linewidth]{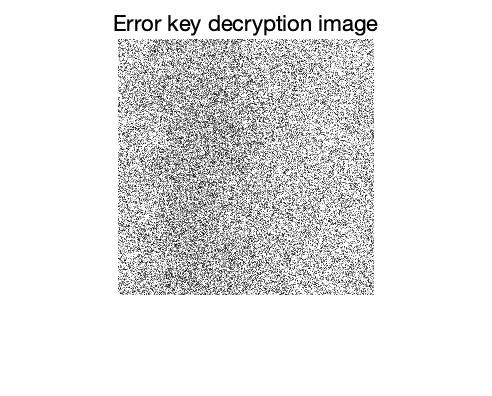} }	
\caption{Simulation of image  encryption and image decryption}
\label{FIG6}%
\end{figure}

 As is shown in Fig. \ref{FIG6}, (a) is the test image with
 256 pixels $\times$ 256 pixels
 of image encryption; (b) is the image after encryption;
(c) is the image after decryption with correct key;
(d), (e), and (f) are the images after decryption with
wrong key, respectively.  The simulation passwords, parameter, and images  of both the image encryption and the image decryption   are presented in the following Table \ref{tab:1}.

\begin{table}[H]
\centering
\begin{tabular}{ccc}
\toprule[1.5pt]
\textbf{Password} & \textbf{Parameters:} $\boldsymbol{\alpha}
=({\alpha _1},{\alpha _2}),\ \beta ,\ \boldsymbol{\gamma}
= ({\theta _1},{\theta _2})$  & \textbf{Image} \\
\midrule[1pt]
correct password &$\boldsymbol{\alpha}  =
(\frac{7\pi }{8},\frac{{5\pi }}{8}),\ \beta
= 0.75,\ \boldsymbol{\gamma}  = (\frac{\pi }{4},
\frac{3\pi }{8})$
& (c) in Fig. \ref{FIG6}\\
error password &$\boldsymbol{\alpha}=
(\frac{7\pi }{8},\frac{{5\pi }}{8}),\ \beta=
0.75,\ \boldsymbol{\gamma}=(\textcolor[rgb]{1, 0, 0}
{\frac{(1+0.05)\pi }{4}},\frac{3\pi }{8})$
& (d) in Fig. \ref{FIG6} \\
error password &$\boldsymbol{\alpha}=
(\textcolor[rgb]{ 1,  0,  0}{\frac{(7+0.1)\pi }{8}},
\frac{5\pi }{8}),\ \ \beta=0.75,\
\boldsymbol{\gamma} =( \frac{\pi }{4},
\frac{3\pi }{8})$    & (e) in Fig. \ref{FIG6} \\
error password &$\boldsymbol{\alpha}=
(\frac{7\pi }{8},\frac{5\pi }{8}),\ \ \beta=
\textcolor[rgb]{ 1,  0,  0}{0.85},\ \boldsymbol{\gamma}=
(\textcolor[rgb]{ 1,  0,  0}{\frac{(1+0.05)\pi }{4}},
\frac{3\pi }{8})$ & (f) in Fig. \ref{FIG6} \\
\bottomrule[1.5pt]
\end{tabular}
\caption{Decryption images corresponding to different keys}
\label{tab:1}
\end{table}

  As is shown   in Fig.
\ref{FIG6} and Table \ref{tab:1}, even if  one
knows that the double phase coding image encryption method based on the
fractional Riesz potential related to chirp functions is used, even if one changes
only one parameter of the correct keys
in (c)  and (d) in Fig. \ref{FIG6}, one cannot obtain the
encrypted original image, let alone we have five parameters.
 In other words, our parameters and symbols provide more degrees
of freedom and make the information more secure.

 The following Fig. \ref{FIG7} presents the   mean square error (for short, MSE)
curves of the decrypted image and the original image when
there are different deviations of  the keys for both the
double phase coding image encryption  based on the FRFT and the
double phase coding image encryption  based on
the fractional Riesz potential related to chirp functions, respectively.

\begin{figure}[H]
\centering
\includegraphics[width=1\linewidth]{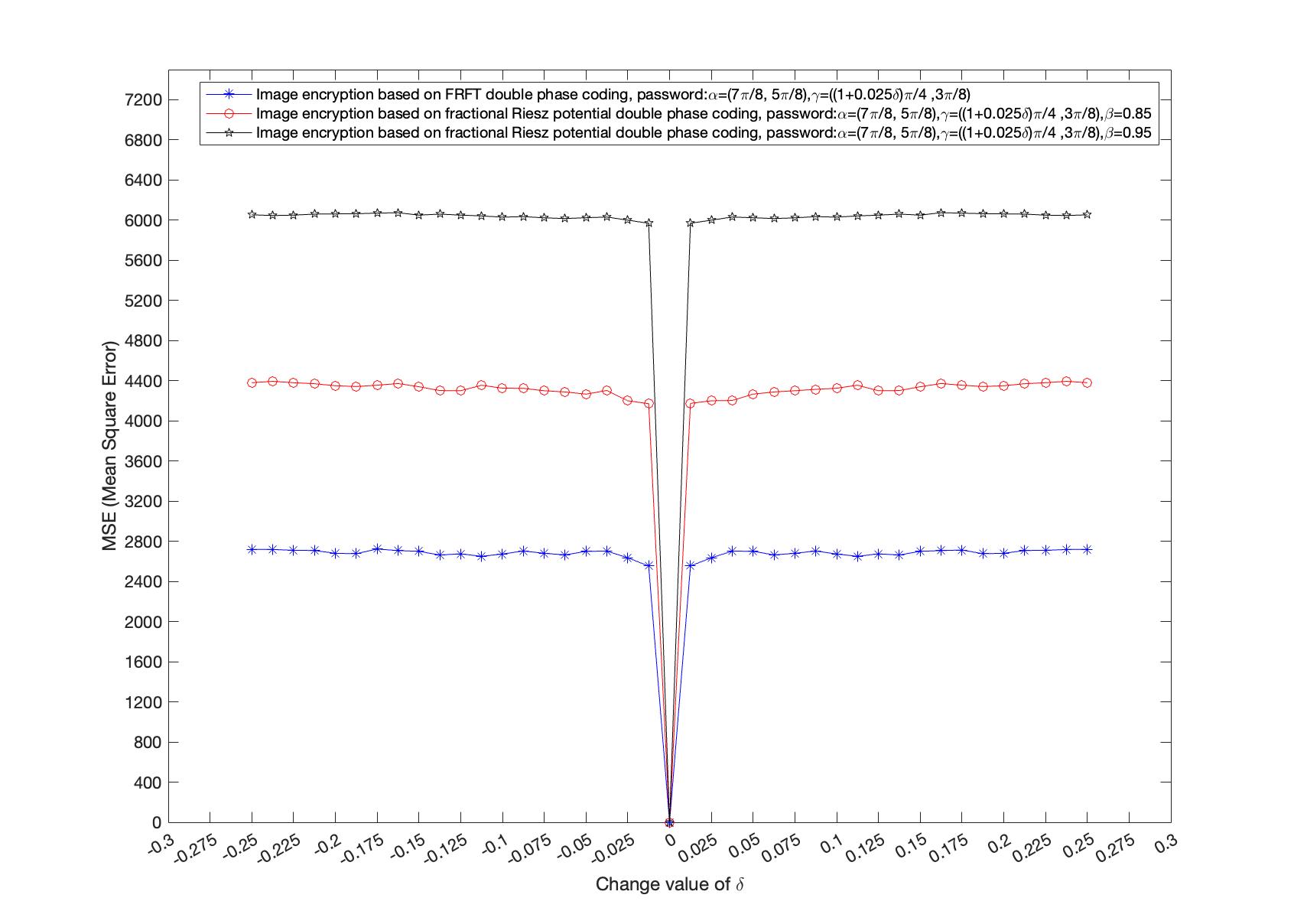}
\caption{When the key is wrong, the MSE curve of the double phase
coded image decryption based on the FRFT and  based on $I_\beta^{\boldsymbol{\alpha}}$ compared to the original image, respectively.}
\label{FIG7}%
\end{figure}

The correct passwords for the image encryption with double
phase coding  based on the FRFT are $\alpha=(7\pi/8, 5\pi/8)$ and
$\gamma=(\pi/4,3\pi/8)$; the correct passwords for the image
encryption with double phase coding based on the fractional
Riesz potentials related to chirp functions are $\alpha=(7\pi/8, 5\pi/8)$,
$\gamma=(\pi/4,3\pi/8)$, and $\beta=0.75$.
From Fig. \ref{FIG7}, we   infer that the parameter
$\beta$ of the double phase coded image encryption
based on the fractional Riesz transform related to chirp functions  greatly
improves the security compared with the double phase coded
image encryption based on the FRFT, and we    also deduce
 that the key is much more secure   when blindly decrypting.

To sum up, the above discussions reveal that the fractional Riesz
potential related to chirp functions can be applied to image encryption,
and the encryption effect is powerful.

\section{Conclusions}\label{sec5}

In this article, we introduce fractional Riesz potentials related to chirp functions,
establish their relations with the FRFT, the fractional Laplace operator
related to chirp functions, and the fractional Riesz transform related to chirp functions.
We apply the fractional Riesz potential related to chirp functions   to the
image encryption. Our experiments show that the symbol
of fractional Riesz potential related to chirp functions essentially expands the key space and
greatly improves the security of images.

\medskip

\bigskip

\noindent Zunwei Fu

\smallskip

\noindent School of Mathematics and Statistics, Linyi University,
Linyi 276000, People's Republic of China; College of Information Technology, The University of Suwon, Hwaseong-si 18323,  South Korea

\smallskip

\noindent{\it E-mail:} \texttt{zwfu@suwon.ac.kr}

\bigskip

 \noindent Yan Lin and Shuhui Yang

\smallskip

\noindent School of Science, China University of Mining and
Technology, Beijing 100083,  People's Republic of China

\smallskip

\noindent{\it E-mails:} \texttt{linyan@cumtb.edu.cn} (Y. Lin)

\noindent\phantom{{\it E-mails:} }\texttt{yangshuhui@student.cumtb.edu.cn} (S. Yang)

\bigskip

\noindent Dachun Yang (Corresponding author)

\smallskip

\noindent  Laboratory of Mathematics and Complex
Systems (Ministry of Education of China),
School of Mathematical Sciences, Beijing Normal
University, Beijing 100875, People's Republic of China

\smallskip

\noindent{\it E-mail:} \texttt{dcyang@bnu.edu.cn}

\end{document}